\documentclass{amsart}
\usepackage{amssymb}
\usepackage{amsbsy}
\usepackage{amsfonts}
\usepackage{amsmath,amscd,latexsym}
\usepackage{multirow}
\usepackage{array}
\usepackage{paralist}
\usepackage{titletoc}
\usepackage{subcaption}

\usepackage[pdftex]{graphicx}
\graphicspath{{Para_Pic/}}

\theoremstyle{plain}
\newtheorem{theorem}{Theorem}[section]
\newtheorem{proposition}[theorem]{Proposition}
\newtheorem{lemma}[theorem]{Lemma}

\newtheorem{remark}[theorem]{Remark}
\newtheorem{definition}[theorem]{Definition}

\newtheorem{main theorem}[theorem]{Main Theorem}

\newcommand{\interior}{\operatorname{int}}

\newcommand{\closure}{\operatorname{cl}}

\newcommand{\Sym}{\operatorname{Sym}}

\newcommand{\NN}{\mathbb{N}}
\newcommand{\ZZ}{\mathbb{Z}}
\newcommand{\QQ}{\mathbb{Q}}
\newcommand{\RR}{\mathbb{R}}

\newcommand{\HH}{\mathbb{H}}
\newcommand{\QQQ}{\hat{\mathbb{Q}}}

\newcommand{\Conway}{\mbox{\boldmath$S$}^{2}}
\newcommand{\Conways}
{(\mbox{\boldmath$S$}^{2},\mbox{\boldmath$P$}^0)}
\newcommand{\PP}{\mbox{{\boldmath$P$}}^0}
\newcommand{\PConway}{\check{\mbox{\boldmath$S$}}^2}
\newcommand{\Ball}{\mbox{\boldmath$B$}^3}

\newcommand{\rtangle}[1]{({\mathbf B}^3,t({#1}))}

\newcommand{\SL}{\operatorname{SL}}
\newcommand{\PGL}{\operatorname{PGL}}
\newcommand{\GL}{\operatorname{GL}}

\newcommand{\Isom}{\operatorname{Isom}}
\newcommand{\Fix}{\operatorname{Fix}}
\newcommand{\Aut}{\operatorname{Aut}}
\newcommand{\Out}{\operatorname{Out}}

\newcommand{\DD}{\mathcal{D}}
\newcommand{\Farey}{\mathcal{F}}

\newcommand{\RGPP}[1]{\hat\Gamma_{#1}}
\newcommand{\RGP}[1]{\Gamma_{#1}}

\newcommand{\OO}{{\mathcal{O}}}
\newcommand{\Piorb}{O}
\newcommand{\orbo}{\mathcal{O}}

\newcommand{\orbm}{\mathcal{M}}

\newcommand{\RotG}{{\mathcal{J}}}

\newcommand{\svert}{\,|\,}

\def\threeball{{{\mathbf B}^3}}

\makeatletter
\renewcommand\subsection{\@startsection{subsection}{2}{0mm}
    {-10.5dd plus-8pt minus-4pt}{10.5dd}
     {\normalsize\upshape}}
\makeatother

\begin{document}
\title[Parabolic generating pairs]
{Classification of parabolic generating pairs of Kleinian groups
with two parabolic generators}

\dedicatory{Dedicated to Professor Bruno Zimmermann on the occasion of his 70-th birthday}

\author{Shunsuke Aimi}
\address{Yonago-Nishi High School\\
Ohtani-cho 200, Yonago, Tottori, 683-8512, Japan}
\email{1991.aimi@gmail.com}

\author{Donghi Lee}
\address{Department of Mathematics\\
Pusan National University \\
San-30 Jangjeon-Dong, Geumjung-Gu, Pusan, 609-735, Korea}
\email{donghi@pusan.ac.kr}

\author{Shunsuke Sakai}
\address{Department of Mathematics\\
Hiroshima University\\
Higashi-Hiroshima, 739-8526, Japan}
\email{d180102@hiroshima-u.ac.jp}

\author{Makoto Sakuma}
\curraddr{Advanced Mathematical Institute\\
Osaka City University\\
3-3-138, Sugimoto, Sumiyoshi, Osaka City
558-8585, Japan}
\address{Department of Mathematics\\
Hiroshima University\\
Higashi-Hiroshima, 739-8526, Japan}
\email{sakuma@hiroshima-u.ac.jp}

\keywords{two bridge link; Heckoid group; parabolic transformation.}
\subjclass[2010]{Primary 57M50, Secondary 57M25}


\begin{abstract}
We give an alternative proof to Agol's classification
of parabolic generating pairs of 
non-free Kleinian groups generated by two parabolic transformations.
As an application, we give a complete characterisation 
of epimorphims between $2$-bridge knot groups
and a complete characterisation of degree one maps 
between the exteriors of hyperbolic $2$-bridge links.
\end{abstract}

\maketitle

\section{Introduction}
\label{sec:intro}

In \cite[Theorem 4.3]{Adams1},
Adams proved that a torsion free Kleinian group of cofinite volume
is generated by two parabolic transformations
if and only if the quotient hyperbolic manifold is
homeomorphic to the complement of a $2$-bridge link
which is not a torus link.
This refines the result of Boileau and Zimmermann \cite[Corollary~3.3]{Boileau-Zimmermann}
that a link in $S^3$ is a $2$-bridge link if and only if its link group is generated by two meridians.

In 2002, Agol \cite{Agol} announced the following classification theorem
of non-free Kleinian groups generated by two parabolic transformations,
which generalises Adams' result.

\begin{theorem}
\label{main-theorem}
A non-free Kleinian group $\Gamma$ is generated by two non-commuting  
parabolic elements if and only if one of the following holds.
\begin{enumerate}[\rm (1)]
\item
$\Gamma$ is conjugate to the hyperbolic $2$-bridge link group, $G(r)\cong \pi_1(S^3-K(r))$,
for some rational number $r=q/p$,
where $p$ and $q$ are relatively prime integers such that 
$q\not\equiv \pm 1 \pmod{p}$.
\item
$\Gamma$ is conjugate to the Heckoid group, $G(r;n)$,
for some $r\in\QQ$ and some $n\in \frac{1}{2}\NN_{\ge 3}$.
\end{enumerate}
\end{theorem}

Adams also proved that each hyperbolic $2$-bridge link groups has only finitely many
distinct parabolic generating pairs up to equivalence
\cite[Corollary 4.1]{Adams1}
and moreover that the figure-eight knot group
has precisely two such pairs up to equivalence
\cite[Corollary 4.6]{Adams1}.
Here, a {\it parabolic generating pair} of a non-elementary Kleinian group
$\Gamma$ is an unordered pair 
of two parabolic transformations 
that generate $\Gamma$.
Two parabolic generating pairs $\{\alpha,\beta\}$
and $\{\alpha',\beta'\}$ of $\Gamma$ are said to be {\it equivalent}
if $\{\alpha',\beta'\}$ is equal to $\{\alpha^{\epsilon_1},\beta^{\epsilon_2} \}$
for some $\epsilon_1, \epsilon_2 \in \{\pm1\}$
up to simultaneous conjugation.

Agol~\cite{Agol} also announced the following theorem,
which generalises the above results of Adams.
(See Figure \ref{generating-pairs} for intuitive description 
and Section \ref{sec:description-theorem} for precise description.)

\begin{theorem}
\label{main-theorem2}
(1) Every hyperbolic $2$-bridge link group $G(r)$
has precisely two parabolic generating pairs up to equivalence.

(2) Every Heckoid group $G(r;n)$ has a unique parabolic generating pair up to equivalence.
\end{theorem}

\begin{figure}
	\centering
	\includegraphics[width=11cm]{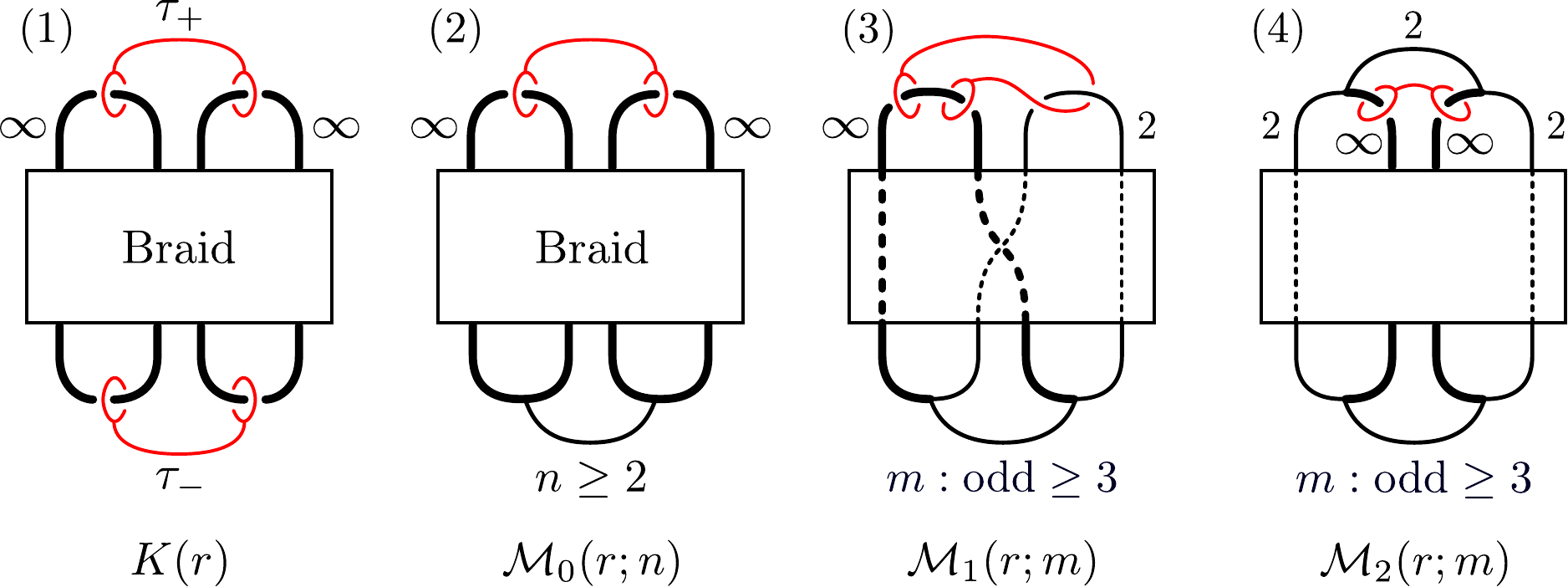}
\caption{
The black graphs illustrate weighted graphs representing $2$-bridge links and Heckoid orbifolds,
where the thick edges with weight $\infty$ correspond to parabolic loci
and thin edges with integral weights represent the singular set.
The red thin graphs represent the parabolic generating pairs of the hyperbolic $2$-bridge link groups and the Heckoid groups.}
	\label{generating-pairs}
\end{figure}

In the last author's joint paper \cite{AOPSY}
with Hirotaka Akiyoshi, Ken'ichi Ohshika, John Parker, and Han Yoshida,
a full proof to Theorem \ref{main-theorem} was given.
The main purpose of this companion paper of \cite{AOPSY}
is to give an alternative \lq homological' proof 
to Theorem \ref{main-theorem2}. 
As an application,
we give a complete characterisation 
of epimorphims between $2$-bridge knot groups
(Theorem \ref{thm:cor-epimorohism} and Remark \ref{rem:determination-epi})
and a complete characterisation of degree one maps 
between the exteriors of hyperbolic $2$-bridge links (Theorem \ref{thm:degree-one}).
Since it is proved by Boileau, Boyer, Reid, and Wang \cite[Corollary~1.3]{BBRW}
that a knot group which is a \lq\lq homomorph'' of a $2$-bridge knot group
is itself a $2$-bridge knot group,
Theorem \ref{thm:cor-epimorohism} 
gives an alternative \lq\lq precise'' proof 
for $2$-bridge knot groups to Simon's conjecture
that a knot group admits epimorphisms onto only finitely many knot groups,
which was established by Agol and Liu \cite{Agol-Liu}
in full generality.
Theorem \ref{thm:cor-epimorohism} has been already used in 
\cite{HMJS, Suzuki, Suzuki-Tran}.
We apologise the delay in writing up the result.

In the preceding paper \cite{Lee-Sakuma_2016} by the second and last authors, 
a proof to Theorem \ref{main-theorem2}(1) 
for the genus-one $2$-bridge knot groups was given
by using the small cancellation theory.
In the first author's master thesis \cite{Aimi} supervised by the last author,
a proof to the result for all hyperbolic $2$-bridge knots was given
by using the Alexander invariants.
In this paper, we give a simple proof for all hyperbolic 
$2$-bridge link groups and for all Heckoid groups, 
by using the homology of the double branched coverings.

It should be noted that 
the parabolicity of the generating pairs is essential
in Theorem \ref{main-theorem2}.
In fact, Heusener and Porti~\cite{Heusener-Porti} proved that
every hyperbolic knot admits infinitely many generating pairs
up to Nielsen equivalence.
Moreover, the same conclusion for torus knots,
which include the non-hyperbolic $2$-bridge knots,
had been proved by Zieschang~\cite{Zieschang}.

\medskip

Our proof of Theorem \ref{main-theorem2} is based on Boileau's suggestion~\cite{Boileau}
to use the well-known fact that, for a parabolic generating pair $\{\alpha, \beta\}$
of $\Gamma$, there is an order $2$ elliptic element $h$ in the normaliser $N(\Gamma)$
of $\Gamma$ in $\Isom^+(\HH^3)$ such that
$(h \alpha h^{-1}, h \beta h^{-1})=(\alpha^{-1}, \beta^{-1})$
(see Propositions \ref{prop:strong-involution_1} and \ref{prop:strong-involution_2}).
Since we can determine 
the isometry groups of the hyperbolic $2$-bridge link complements and the Heckoid orbifolds
(Propositions \ref{prop:isometry-group} and \ref{prop:isometry-group-Heckoid}), 
Boileau's suggestion leads us to a finite list of possible parabolic generating pairs
(see Propositions \ref{prop:strong-invertion_1}, \ref{prop:strong-invertion_2},
and Section \ref{sec:Heckoid-orb}).
For the hyperbolic $2$-bridge link groups,
we can exclude all fake parabolic generating pairs through 
simple calculations on the homology of the double branched coverings
(Sections \ref{sec:long-meridian-pair} and \ref{sec:extra-meridian-pair}).
For the Heckoid groups,
we can also do so through a simple argument by using the orbifold theorem
and by using natural epimorphisms from Heckoid groups onto 
the $\pi$-orbifold groups of $2$-bridge links
(Section \ref{sec:Heckoid-orb}).
This is the strategy of our proof of the main Theorem \ref{main-theorem2}.

\medskip

At the end of the introduction, 
we note that Agol~\cite{Agol} obtained Theorem \ref{main-theorem2}(1)
as a corollary of the following much stronger theorem.

\begin{theorem}
\label{Agol-original-theorem}
For any hyperbolic $2$-bridge link group $G(r)$
and for any meridian pair in $G(r)$
which is not equivalent to the upper nor lower
meridian pair,
the subgroup of $G(r)$ generated by the meridian pair
is a free group.
\end{theorem}

He proved this theorem
through a beautiful ping-pong argument applied to the ideal boundary
of the universal covering space of the natural local $\mathrm{CAT(0)}$ cubical complex 
of the $2$-bridge link complement,
which is constructed from a reduced alternating diagram
of the $2$-bridge link.

\medskip

This paper is organised as follows.
In Section \ref{sec:description-theorem},
we give a detailed description of Theorem \ref{main-theorem2}.
In Section \ref{sec:inverting-elliptic-element},
we recall Boileau's key suggestion (Proposition \ref{prop:strong-involution_1})
which relates parabolic generating pairs of hyperbolic $2$-bridge links
to strong inversions of the links,
and present its extension to Heckoid orbifolds
(Proposition \ref{prop:strong-involution_2}).
In Section \ref{sec:two-bridge-link-group},
we classify the strong inversions of hyperbolic $2$-bridge links
and list all possible parabolic generating pairs of the link groups
(Propositions \ref{prop:strong-invertion_1} and \ref{prop:strong-invertion_2}).
In Sections \ref{sec:long-meridian-pair} and \ref{sec:extra-meridian-pair},
we show that all possible parabolic generating pairs,
except for the upper and lower meridian pairs, 
are not generating pairs,
by using the homology of the double branched coverings
(Propositions \ref{prop:long-meridian-pair} and \ref{prop:extra-meridian-pair}).
Thus the proof of the first assertion of Theorem \ref{main-theorem2}
is completed in this section.
In Section \ref{sec:Heckoid-orb},
we prove the second assertion of Theorem \ref{main-theorem2},
after determining the orientation-preserving isometry groups 
of Heckoid orbifolds (Proposition \ref{prop:isometry-group-Heckoid})
by using \cite[Proposition 12.6]{AOPSY}.
In Section \ref{sec:application},
we give a characterisation 
of epimorphims between $2$-bridge knot groups
(Theorem \ref{thm:cor-epimorohism})
and
a characterisation of degree one maps 
between the exteriors of hyperbolic $2$-bridge links
(Theorem \ref{thm:degree-one}),
by using the main Theorem \ref{main-theorem2}.

\medskip
{\bf Acknowledgement.}
The last author would like to thank Ian Agol 
for sending the slide of his talk \cite{Agol},
encouraging him (and any of his collaborators) to write up the proof,
and describing key ideas of the proof.
He would also like to thank Michel Boileau
for enlightening conversation.
His sincere thanks also go to all the other authors 
for joining the project to give a proof to Agol's announcement.
The authors would like to thank the referee for his/her very careful reading
and for suggesting the possibility of characterisation
of degree one maps between the exteriors of hyperbolic $2$-bridge links
(Theorem \ref{thm:degree-one}).
The second author was supported by Basic Science Research Program
through the National Research Foundation of Korea(NRF) funded
by the Ministry of Education, Science and Technology(2017R1D1A1B03029904).
The fourth author was supported by JSPS Grants-in-Aid 15H03620, 20K03614,
and by Osaka City University Advanced Mathematical Institute (MEXT Joint Usage/Research Center on Mathematics and Theoretical Physics JPMXP0619217849).

\section{Precise description of Theorem \ref{main-theorem2}}
\label{sec:description-theorem}

We first give a precise description of Theorem \ref{main-theorem2}(1).
For a rational number $r=q/p$, let $K(r)$ be the $2$-bridge link of slope $r$.
Recall that $K(r)$ is hyperbolic if and only if $q\not\equiv \pm 1 \pmod{p}$.
When $r$ satisfies this condition,
the {\it hyperbolic $2$-bridge link group} $G(r)$ is defined to be
the Kleinian group, which is unique up to conjugation,
such that $\HH^3/G(r)$ is homeomorphic to $S^3-K(r)$
as an oriented manifold.
Thus $G(r)$ is isomorphic to $\pi_1(E(K(r)))$,
where $E(K(r))=S^3-\interior N(K(r))$ is the {\it exterior} 
(the complement of an open regular neighbourhood) of $K(r)$.
The Kleinian group $G(r)$ is generated by two parabolic transformations, and
it is proved by
Adams~\cite[Theorem 4.3]{Adams1} that 
any parabolic generating pair 
of $G(r)$ consists of meridians.
Here an element of $G(r)\cong\pi_1(E(K(r)))$ is called a {\it meridian}
if it is freely homotopic to a meridional loop in $\partial E(K(r))$,
i.e. a simple loop in $\partial E(K(r))$ that bounds an essential disc in $N(K(r))$.
This implies that any parabolic generating pair $\{\alpha,\beta\}$ is represented 
by an arc properly embedded in $E(K(r))$,
together with a pair of meridional loops on $\partial E(K(r))$ passing through
the endpoints of the arc.
(See \cite[Section 2]{Lee-Sakuma_2016} for more detailed explanation.)
It is obvious that 
the meridian pair represented by the upper tunnel $\tau_{+}$
(resp. the lower tunnel $\tau_{-}$)
forms a parabolic generating pair of $G(r)\cong \pi_1(E(K(r)))$,
and it is called the {\it upper meridian pair}
(resp. {\it lower meridian pair}) of $\pi_1(E(K(r)))$.
See Figure \ref{generating-pairs}(1), where the black bold graph represent
a $2$-bridge link, and the two red thin graphs represent the upper and lower
meridian pairs, respectively. 
The \lq weight' $\infty$ of the $2$-bridge link 
indicates that we consider the link exterior and that its boundary tori
form the parabolic locus of the hyperbolic structure.
For precise definitions of
$2$-bridge links and upper/lower tunnels,
please see the companion paper \cite[Section 2]{AOPSY}.
{\it Theorem \ref{main-theorem2}(1) says that,
for each hyperbolic $2$-bridge link group $G(r)$,
(i) the upper meridian pair and the lower meridian pair 
are the only parabolic generating pairs of $G(r)$,
and (ii) they are not equivalent.}

\begin{remark}
{\rm
The upper/lower tunnels are actually {\it unknotting tunnels},
namely the complement of an open regular neighbourhood of the tunnel
is a genus $2$ handlebody. 
Moreover, the upper (resp. lower)
meridian pair is the image of a free generating pair
of the fundamental group of the complementary handlebody of 
the lower (resp. upper) tunnel.
Note that we do not know a priori that 
the tunnel associated to a parabolic generating pair is an unknotting tunnel.
}
\end{remark}


Next, we give a precise description of Theorem \ref{main-theorem2}(2).
The Heckoid groups were first introduced by Riley~\cite{Riley1992}
as an analogy of the classical Hecke group,
and it was reformulated by Lee and Sakuma \cite{Lee-Sakuma_2013}
following Agol \cite{Agol}
as the orbifold fundamental groups of the Heckoid orbifolds
illustrated in Figure \ref{generating-pairs}(2)-(4).
(See \cite{BMP, Boileau-Porti, CHK} for basic terminologies and facts
concerning orbifolds.)
These figures illustrate weighted graphs $(S^3,\Sigma, w)$
whose explicit descriptions are given by Definition \ref{def:Heckoid-orbifold}.
For each weighted graph $(S^3,\Sigma, w)$ in the figure,
let $(M_0,P)$ be the pair of a compact $3$-orbifold $M_0$
and a compact $2$-suborbifold $P$ of $\partial M_0$
determined by the rules described below.
Let $\Sigma_{\infty}$ 
be the subgraph of $\Sigma$ consisting of the edges 
with weight $\infty$, 
and let $\Sigma_s$ be the subgraph of $\Sigma$ consisting of the edges 
with integral weight.
\begin{enumerate}
\item
The underlying space $|M_0|$ of $M_0$ is the complement of an open regular neighbourhood of 
the subgraph $\Sigma_{\infty}$.
\item
The singular set of $M_0$ is 
$\Sigma_0:=\Sigma_s\cap |M_0|$, 
where the index of each edge of the singular set is given by the weight $w(e)$
of the corresponding edge $e$ of $\Sigma_s$.\item
For an edge $e$ of $\Sigma_\infty$,
let $P$ be the $2$-suborbifold of $\partial M_0$ defined as follows.
\begin{enumerate}
\item
In Figure \ref{generating-pairs}(2), $P$ consists of 
two annuli in $\partial M_0$
whose cores, respectively, are meridians of the two edges of $\Sigma_{\infty}$.
\item
In Figure \ref{generating-pairs}(3), $P$ consists of an annulus in $\partial M_0$
whose core is a meridian of the single edge of $\Sigma_{\infty}$.
\item
In Figure \ref{generating-pairs}(4), $P$ consists of two copies of the
annular orbifold 
$D^2(2,2)$ (the $2$-orbifold with underlying space the disc
and with two cone points of index $2$)
in $\partial M_0$
each of which is 
bounded by a meridian of an edge of $\Sigma_{\infty}$.
\end{enumerate}
\end{enumerate}

By \cite[Lemmas 6.3 and 6.6]{Lee-Sakuma_2013},
the orbifold pair $(M_0,P)$ is a Haken pared orbifold
(see \cite[Definition 8.3.7]{Boileau-Porti}) 
and admits a unique complete hyperbolic structure,
which is geometrically finite 
(see \cite[Proposition 6.7]{Lee-Sakuma_2013}, \cite[Section 3]{AOPSY}).
Namely there is a geometrically finite Kleinian group $\Gamma$,
unique up to conjugation,
such that $M=\HH^3/\Gamma$ is homeomorphic to the 
interior of the compact orbifold $M_0$,
such that $P$ represents to the parabolic locus.
To be more precise, there are positive constants $\delta$ and $\mu$, such that
\begin{align}
\label{formula:convex-core}
(M_0,P) \cong 
(\mathrm{thick}_{\mu}(C_{\delta}(M)),
\partial(\mathrm{thick}_{\mu}(C_{\delta}(M)))\cap 
(\mathrm{thin}_{\mu}(C_{\delta}(M))),
\end{align}
where $C_{\delta}(M)$
is the closed $\delta$-neighbourhood of the convex core $C(M)$ of $M$, and 
$\mathrm{thick}_{\mu}(C_{\delta}(M))$ and 
$\mathrm{thin}_{\mu}(C_{\delta}(M))$ are 
the $\mu$-thick part and $\mu$-thin part
(see \cite{Boileau-Porti} for terminology).
The $\mu$-thin part $\mathrm{thin}_{\mu}(C_{\delta}(M))$
consists of only cuspidal components and 
it is isomorphic to $P\times [0,\infty)$.
We denote by $P_{\alpha}$ and $P_{\beta}$ the components of $P$
corresponding to the parabolic transformations $\alpha$ and $\beta$, respectively,
i.e., $P_{\alpha}$ (resp. $P_{\beta}$) is the component
of $P$ which is the image of a subsurface of an 
$\alpha$-invariant (resp. $\beta$-invariant) horosphere.
The pair $(M_0,P)$ is also regarded as a relative compactification
of the pair consisting of a non-cuspidal part of $M$ and its boundary
(see \cite[Section 5]{AOPSY}).

We denote the pared orbifold $\orbm:=(M_0,P)$ by 
$\orbm_{0}(r;n)$, $\orbm_{1}(r;m)$, or $\orbm_{2}(r;m)$
according as it is described by the weighted graph in 
Figure \ref{generating-pairs}(2), (3), or (4).
We also denote the Kleinian group $\Gamma$
that uniformises the pared orbifold $\orbm$ by $\pi_1(\orbm)$.
(More generally, we use the symbol $\pi_1(\cdot)$ to denote
the orbifold fundamental group.)
For $r=q/p\in\QQ$ and $n\in \frac{1}{2}\NN_{\ge 3}$,
the {\it Heckoid group} $G(r;n)$ is the Kleinian group
that is defined as follows.
\[
G(r;n) \cong 
\begin{cases}
\pi_1(\orbm_{0}(r;n))
& \text{if $n\in \NN_{\ge 2}$,}\\
\pi_1(\orbm_{1}(\hat r;m))
& \text{if $n=m/2$ for some odd $m\ge 2$ and if $p$ is odd,}\\
\pi_1(\orbm_{2}(\hat r;m))
& \text{if $n=m/2$ for some odd $m\ge 2$ and if $p$ is even,}\\
\end{cases}
\]
where $\hat r$ is defined from $r=q/p$ by the following rule.
\[
\hat r
=
\begin{cases}
\frac{q/2}{p}
& \text{if $p$ is odd and $q$ is even,}\\
\frac{(p+q)/2}{p}
& \text{if $p$ is odd and $q$ is odd,}\\
\frac{q}{p/2}
& \text{if $p$ is even.}
\end{cases}
\]
For the background of this rather complicated definition,
please see \cite{Lee-Sakuma_2016, AOPSY}.

By \cite[p.197]{Adams1} (cf. \cite[Lemma 7.2]{AOPSY}),
any parabolic generating pair $\{\alpha,\beta\}$ of the Heckoid group $G(r;n)$
consists of primitive elements.
Since the parabolic locus $P_{\alpha}\cong P_{\beta}$ is an annular orbifold 
$S^1\times I$ or $D^2(2,2)$, this implies that
$\alpha$ and $\beta$ are freely homotopic to the unique (up to isotopy)
essential simple loop in $P_{\alpha}$ and $P_{\beta}$, respectively.
Thus they are freely homotopic to a meridional loop 
of the corresponding edge of $\Sigma_{\infty}$.
Since $G(r;n)$
is identified with a quotient of
the usual fundamental group $\pi_1(|M_0|-\Sigma_0)$,
it follows that
any parabolic generating pair is represented by
a graph embedded in $|M_0|-\Sigma_0$
consisting of two meridional loops of $\Sigma_{\infty}$
and an arc joining them.
It follows from the definition of Heckoid groups that
the pairs of parabolic elements represented by 
the red graphs in Figure \ref{generating-pairs}(2)-(4) 
are generating pairs (see \cite[Section 3]{Lee-Sakuma_2013}, \cite[Proposition 3.3]{AOPSY}).
{\it Theorem \ref{main-theorem2}(2) says that
each Heckoid group $G(r;n)$ admits a unique parabolic generating pair,
and it is illustrated by Figure \ref{generating-pairs}
(2), (3) or (4) 
according to the type of $G(r;n)$.}

\medskip

In the special case where $r\in\ZZ$, the Heckoid group $G(r;n)$ 
with $n=\frac{m}{2}\in\frac{1}{2}\NN_{\ge 3}$
is conjugate to $G(0;n)$,
and it is a Fuchsian group,
which is essentially equal to the classical {\it Hecke group}
$H(m)\cong\pi_1(S^2(2,m,\infty))$, generated by
the following matrices (\cite{Hecke}):
\[
A_m:= \begin{pmatrix} 1 & 2\cos\frac{\pi}{m} \\ 0 & 1 \end{pmatrix}
\quad \mbox{and} \quad
Q:=\begin{pmatrix} 0 & 1 \\ -1 & 0 \end{pmatrix}
\]
To be precise, $G(r;n)\cong G(0;n)$
is conjugate to the subgroup of $H(m)$, with $n=\frac{m}{2}$,
generated by $\{A_m,QA_mQ^{-1}\}$.
Thus according to whether (a) $n$ is an integer $\ge 2$
or (b) a half-integer $n=m/2$ ($m$: odd $\ge 3$),
the Heckoid group $G(0;n)$ is conjugate to
(a) the index $2$ subgroup of the Hecke group $H(m)$
isomorphic to $\pi_1(S^2(n,\infty,\infty))$ 
or 
(b) the Hecke group $H(m)$, with $m=2n$.
These follow from the following topological observation.
\begin{enumerate}
\item
For every $r\in\ZZ$ and $n\in \NN_{\ge 2}$, we have
$\orbm_{0}(r;n)\cong \orbm_{0}(0;n)\cong S^2(n,\infty,\infty)\times I$,
and $S^2(n,\infty,\infty)$ is a double orbifold covering of
$S^2(2,m,\infty)$ with $m=2n$.
\item
For every $r\in\ZZ$ and an odd integer $m\in \NN_{\ge 3}$, we have
$\orbm_{1}(\hat r;m)\cong \orbm_{1}(0;m)\cong S^2(2,m,\infty)\times I$.
\end{enumerate}

In \cite[Proposition 4.2]{Knapp},
Knapp completely determined when a given pair of non-commuting parabolic transformations
of the hyperbolic plane generate a discrete group,
by using Poincar\'e's theorem on fundamental polyhedra.
In our terminology, it is described as follows,
which particularly implies Theorem \ref{main-theorem2}(2) 
for the special case where $r\in\ZZ$. 

\begin{proposition}
\label{prop:classification-fuchsian-case} 
\cite[Proposition 4.2]{Knapp}
A non-elementary non-free Fuchsian group is generated by two parabolic transformations,
if and only if it is conjugate to 
the Fuchsian group $G(0;n)$ for some $n\in \frac{1}{2}\NN_{\ge 3}$.
Moreover, 
$\{A_m,QA_mQ^{-1}\}$, where $m=2n$,
is the unique parabolic generating pair of
$G(0;n)$, up to equivalence.
\end{proposition}

\section{Inverting elliptic elements for 
the parabolic generating pairs}
\label{sec:inverting-elliptic-element}

In this section, we recall Boileau's key suggestion \cite{Boileau},
and present its slight extension.
His suggestion is to use the following well-known fact.
Let 
$\Gamma =\langle \alpha, \beta\rangle$ be a non-elementary Kleinian group 
generated by two parabolic transformations $\alpha$ and $\beta$,
and $M=\HH^3/\Gamma$ the quotient hyperbolic $3$-manifold or orbifold.
Let $\eta$ be the geodesic 
joining the parabolic fixed points of $\alpha$ and $\beta$,
and let $h$ be the $\pi$-rotation around $\eta$.
Then we have
\begin{align}
\label{formula:involution}
(h\alpha h^{-1},h\beta h^{-1})=(\alpha^{-1},\beta^{-1}).
\end{align}
We call $h$ the {\it inverting elliptic element} for 
the parabolic generating pair $\{\alpha,\beta\}$ of 
the Kleinian group $\Gamma$.
If $h\notin \Gamma$, then $h$ induces an isometric involution on $M$,
whereas if $h\in\Gamma$, then $M$ is an orbifold with non-empty singular set
and the axis $\eta$ projects to a \lq geodesic' contained in the subset of the 
singular set consisting of even degree edges.
(By a geodesic (arc) in a complete hyperbolic orbifold,
we mean the image of a geodesic (arc) in the universal cover $\HH^3$ in the orbifold.)

We first treat the case where 
$\Gamma$ is a hyperbolic $2$-bridge link group. 
We prepare some terminologies.
For a link $K$ in $S^3$, 
a {\it tunnel} for $K$ is an arc, $\tau$, in $S^3$
such that $\tau\cap K=\partial \tau$.
We assume $\tau$ intersects a fixed regular neighbourhood $N(K)$ of $K$
in two arcs, each of which forms a radius of a meridian disc of $N(K)$.
Then the intersection of $\tau$ with the exterior $E(K)=S^3-\interior N(K)$
is an arc properly embedded in $E(K)$.
Thus, as described in the first paragraph of Section \ref{sec:description-theorem}
(cf. \cite[Section 2]{Lee-Sakuma_2016}),
it determines a pair of meridians in the link group $\pi_1(E(K))$, up to equivalence.
We call it the {\it meridian pair determined by the tunnel $\tau$}.
Here 
two meridian pairs $\{m_1,m_2\}$ and $\{m_1',m_2'\}$ of $\pi_1(E(K))$
is said to be {\it equivalent}, if $\{m_1',m_2'\}$ is
equal to $\{m_1^{\epsilon_1}, m_2^{\epsilon_2}\}$ 
for some $\epsilon_1, \epsilon_2 \in \{\pm1\}$
up to simultaneous conjugation.

A {\it strong inversion} of $K$ is 
an orientation-preserving involution,
which we often denote by the symbol $h$, of
$S^3$ preserving $K$ setwise
such that the fixed point set $\Fix(h)$ is a circle
intersecting each component of $K$ in two points.
Note that $\Fix(h)$ consists of $2\mu$ tunnels,
where $\mu$ is the number of components of $K$.
Then the following proposition,
which holds a key to the proof of 
Theorem \ref{main-theorem2}(1),
was suggested by Boileau~\cite{Boileau}
(cf. \cite[Proposition 2.1]{Lee-Sakuma_2012}).
(See \cite{Adams-Reid, Bleiler-Moriah, Futer} for interesting related results.)

\begin{proposition}
\label{prop:strong-involution_1}
Let $K(r)$ be a hyperbolic $2$-bridge link,
and let $\{\alpha,\beta\}$ be a
parabolic generating pair of the link group $G(r)$.
Then there is a strong inversion, $h$, of $K(r)$
such that 
$\{\alpha,\beta\}$ is a meridian pair
represented by a tunnel contained in $\Fix(h)$.
\end{proposition}

\begin{proof}
Though the proof is given in \cite{Lee-Sakuma_2012},
we recall the proof as a warm-up for the treatment of Heckoid groups.
By definition, $G(r)$ is a non-elementary Kleinian group generated by
the parabolic transformations  $\alpha$ and $\beta$.
Let $h$ be the inverting elliptic element for 
the parabolic generating pair $\{\alpha,\beta\}$ of $G(r)$.
Since $h$ is an element of the normaliser of $G(r)$ in $\Isom^+(\HH^3)$
which does not belong to $G(r)$ (recall that $G(r)$ is torsion free),
$h$ descends to an orientation-preserving involution, $\bar h$,
of $\HH^3/G(r)\cong S^3-K(r)$.
Since each component of $\partial E(K(r))$
corresponds to one of the parabolic loci $P_{\alpha}$ and $P_{\beta}$
(which in turn follows from the fact that $\alpha$ and $\beta$
generate $H_1(E(K(r))$),
it follows that
the restriction of $\bar h$ to each of the components of
$\partial E(K(r))$ is a hyper-elliptic involution.
Hence $\bar h$ extends to an involution of the pair $(S^3,K(r))$,
which we continue to denote by $\bar h$.
Then $\bar h$ is a strong inversion of $K(r)$,
and $\Fix(\bar h)$ contains the image $\bar \eta$ of
the axis $\eta$ of the inverting elliptic element $h$.
Then, by \cite[Theorem 4.3]{Adams1},
$\{\alpha,\beta\}$ is the meridian pair
represented by the tunnel for $K(r)$ 
that is obtained as the closure of $\bar\eta$ in $\Fix(\bar h)$.
By denoting the strong inversion $\bar h$ by $h$,
we obtain the desired result.
\end{proof}

We next treat the case where $\Gamma$ is a Heckoid group $G(r;n)$.
Let $\orbm=(M_0,P)$ be the corresponding Heckoid orbifold,
and identify $M_0$ and $P$ with a subspace of the quotient hyperbolic orbifold 
$M=\HH^3/\Gamma$ through the isomorphism (\ref{formula:convex-core})
in the introduction.
As is noted in the introduction
(see \cite[p.197]{Adams1}, \cite[Lemma 7.2]{AOPSY}),
any parabolic generating pair $\{\alpha,\beta\}$ of $G(r;n)$
consists of primitive elements, and 
$\alpha$ and $\beta$ are freely homotopic to the unique (up to isotopy)
essential simple loops in the annular orbifolds $P_{\alpha}$ and $P_{\beta}$, respectively.
Thus the equivalence class of $\{\alpha,\beta\}$ is uniquely determined
by the proper geodesic arc $\bar\eta\cap M_0$,
where $\bar\eta$ is the \lq geodesic' in the orbifold $M$
obtained as the image of 
the axis $\eta$ of the inverting elliptic element $h$.
(Note that if $\eta$ intersects orthogonally the axis of an
even order elliptic transformation in $\Gamma$,
then the underlying space of $\bar\eta$ has an endpoint
in an even degree singular locus 
of the orbifold $M$.)
If $h\notin\Gamma$, then it induces an isometric involution, $\bar h$, on $M$,
and $\bar\eta$ is contained in $\Fix(\bar h)$.
If $h\in\Gamma$
then $\bar\eta$ is a \lq geodesic edge path' contained 
in the subset of the singular set of $M$
consisting of even degree edges. 
Thus we have the following analogy of Proposition \ref{prop:strong-involution_1},
which holds a key to the proof of 
Theorem \ref{main-theorem2}(2).

\begin{proposition}
\label{prop:strong-involution_2}
Consider the Heckoid group $\Gamma=G(r;n)$,
and let $\orbm=(M_0,P)$ be the corresponding Heckoid orbifold.
Let $\{\alpha,\beta\}$ be a parabolic generating pair of $G(r;n)$,
$h$  the inverting elliptic element for 
the parabolic generating pair $\{\alpha,\beta\}$,
and $\eta$ the axis of $h$,
and $\bar\eta$ be the image of $\eta$ in $M=\HH^3/\Gamma$.
Then $\{\alpha,\beta\}$ is represented by $\bar\eta\cap M_0$.
Moreover the following holds.
\begin{enumerate}
\item
If $h\notin \Gamma$, then it descends to an involution, $\bar h$, of $M$,
such that $\bar\eta$ is contained in $\Fix(\bar h)$.
\item
If $h\in \Gamma$, then $\bar\eta\cap M_0$ is
a geodesic edge path
joining $P_{\alpha}$ and $P_{\beta}$,
contained in the subset of the singular set of $M_0$
consisting of even degree edges. 
\end{enumerate}
\end{proposition}

\section{Symmetries of $2$-bridge links and all possible generating meridian pairs
for $2$-bridge link groups}
\label{sec:two-bridge-link-group}

Let $K(r)$ with $r=q/p$ ($q\not\equiv \pm1\pmod{p}$)  be a hyperbolic $2$-bridge link,
and let $\tau_+$ and $\tau_-$ be the upper and lower tunnel for $K(r)$
(see \cite[Section 2]{AOPSY} for an explicit definition).
In this section, we describe all strong inversions of $K(r)$,
up to strong equivalence, 
and list all possible generating meridian pairs of the 
$2$-bridge link group $G(r)$ by using 
Proposition \ref{prop:strong-involution_1}.
Here two strong inversions of a link $K$ are said to be
{\it strongly equivalent}
if they are conjugate by a homeomorphism of $(S^3,K)$
that is pairwise isotopic to the identity.
To this end, we first describe the orientation-preserving isometry group
$\Isom^+(S^3-K(r))$ of the hyperbolic manifold $S^3-K(r)$.

\begin{proposition}
\label{prop:isometry-group}
For a hyperbolic
$2$-bridge link $K(r)$ with $r=q/p$ ($q\not\equiv \pm1\pmod{p}$),
the orientation-preserving isometry group
$\Isom^+(S^3-K(r))$ is given by the following formula.
\[
\Isom^+(S^3-K(r))\cong
\begin{cases}
(\ZZ_2)^2 & \text{if $q^2\not\equiv 1\pmod{p}$}\\
D_4\cong (\ZZ_2)^2\rtimes \ZZ_2 & \text{if $p$ is odd and $q^2 \equiv 1\pmod{p}$, or}\\
         & \text{if $p$ is even and $q^2 \equiv p+1\pmod{2p}$}\\
(\ZZ_2)^3\cong (\ZZ_2)^2\times \ZZ_2 & \text{if $p$ is even and $q^2 \equiv 1\pmod{2p}$}
\end{cases}
\]
Here $D_4$ denotes the order $8$ dihedral group,
which is regarded as a semi-direct product of $(\ZZ_2)^2$ and $\ZZ_2$.
The subgroups $(\ZZ_2)^2$ in the formula consist of those isometries which preserve 
both the upper tunnel $\tau_+$ and 
the lower tunnel $\tau_-$ setwise:
the elements in $\Isom^+(S^3-K(r))$ which are not contained 
in the characteristic subgroup $(\ZZ_2)^2$
interchange $\tau_+$ and $\tau_-$.
\end{proposition}

\begin{proof}
By the work of Epstein and Penner~\cite{Epstein-Penner},
every cusped hyperbolic manifold $M$ of finite volume
admits a canonical decomposition into hyperbolic ideal polyhedra
and the isometry group of $M$ is isomorphic to
the combinatorial automorphism group of the canonical decomposition.
In \cite{Sakuma-Weeks},
a topological candidate, $\DD$, of the canonical decomposition
of the cusped hyperbolic manifold $S^3-K(r)$ was constructed
\cite[Theorem II.2.5]{Sakuma-Weeks}, and 
the combinatorial automorphism group $\Aut(\DD)$ of
$\DD$ was calculated \cite[Theorem II.3.2]{Sakuma-Weeks}.
(As is noted in \cite{Millichap-Worden}, 
there is an error in the proof of Theorem II.3.2,
but this does not affect the conclusion of the theorem.)
On the other hand, it was proved by Gu\'eritaud~\cite{Gueritaud} 
(cf. \cite{ASWY, Gueritaud-Futer})
that $\DD$ is combinatorially equivalent to the canonical decomposition.
Hence $\Isom^+(S^3-K(r))$ is isomorphic to the
orientation-preserving subgroup $\Aut^+(\DD)$ of $\Aut(\DD)$,
which is described in \cite[Theorem II.3.2]{Sakuma-Weeks}.
\end{proof}

\begin{remark}
{\rm 
The action of $\Isom^+(S^3-K(r))$ on $S^3-K(r)$ 
extends to an action on $(S^3,K(r))$, and
this fact implies that $\Isom^+(S^3-K(r))$ is isomorphic to
the orientation-preserving symmetry group,
$\Sym^+(S^3,K(r))$, 
the group of diffeomorphisms of the pair $(S^3,K(r))$,
up to pairwise isotopy,
which preserves the orientation of $S^3$.
(In fact, it is observed by Riley~\cite[p.124]{Riley1979} that this holds
for all hyperbolic knots in $S^3$.)
We note that the symmetry groups of $2$-bridge links 
had been already obtained, as described below.
For $2$-bridge knots, it is reported in \cite{GLM} that
Conway calculated the outer-automorphism groups of their knot groups,
which is isomorphic to the  full symmetry groups of the knots.
Bonahon and Siebenmann
\cite{Bonahon-Siebenmann} calculated the  full symmetry group
of every $2$-bridge link,
by using the uniqueness of 
$2$-bridge decompositions up to isotopy established by Schubert~\cite{Schubert}.
Another calculation is given by \cite[Theorems 4.1 and 6.1]{Sakuma2}
by using the orbifold theorem.
(As is noted in \cite[the paragraph preceding Proposition 12.5]{AOPSY}, 
though there are misprints in the statement \cite[Theorem 4.1]{Sakuma2},
the correct formula can be found in the tables in 
\cite[p.184]{Sakuma2}.)
}
\end{remark}

We now visualise the action of 
$\Isom^+(S^3-K(r))\cong \Sym^+(S^3,K(r))$ on $(S^3,K(r))$
and classify the strong inversions of $K(r)$,
up to strong equivalence, 
generalising and refining the result for the knot case
given by \cite[Proposition 3.6]{Sakuma1}.
To this end, note that, by virtue of Schubert's classification of $2$-bridge links 
(cf. \cite[Chapter 12]{BZH}),
we may assume that the slope $r=q/p$ of the hyperbolic $2$-bridge link $K(r)$
satisfies the inequality $0<q\le p/2$
and the condition that one of $p$ and $q$ is even.
The following lemma is well-known and can be proved 
by an argument similar to that in \cite[Lemma 2]{Hartley-Kawauchi}.

\begin{lemma}
\label{lem:conti-fract-expansion1}
Let $p$ and $q$ be relatively prime integers
such that $0<q\le p/2$
and that one of $p$ and $q$ is even.
Then $r=q/p$ has a unique 
continued fraction expansion 
\begin{center}\begin{picture}(230,70)
\put(0,48){$\displaystyle{
r=[2b_1,2b_2, \cdots,2b_{n}] =
\cfrac{1}{2b_1+
\cfrac{1}{ \raisebox{-5pt}[0pt][0pt]{$2b_2 \, + \, $}
\raisebox{-10pt}[0pt][0pt]{$\, \ddots \ $}
\raisebox{-12pt}[0pt][0pt]{$+ \, \cfrac{1}{2b_{n}}$}
}} \ ,}$}
\end{picture}\end{center}
where $b_i$ is a non-zero integer $(1\le i\le n)$.
The length $n$ is even or odd
according to whether $p$ is odd or even,
i.e., $K(r)$ is a knot or a two-component link.
Moreover the following hold.
\begin{enumerate}
\item
Suppose $p$ is odd, i.e., $n$ is even.
Then $q^2\equiv 1 \pmod{p}$ if and only if $b_i=-b_{n+1-i}$ ($1\le i\le n$).
\item
Suppose $p$ is even, i.e., $n$ is odd.
Then $q^2\equiv 1 \pmod{2p}$
if and only if $b_i=b_{n+1-i}$ ($1\le i\le n$).
\end{enumerate}
\end{lemma}

\begin{figure}
	\centering
	\includegraphics[width=9cm]{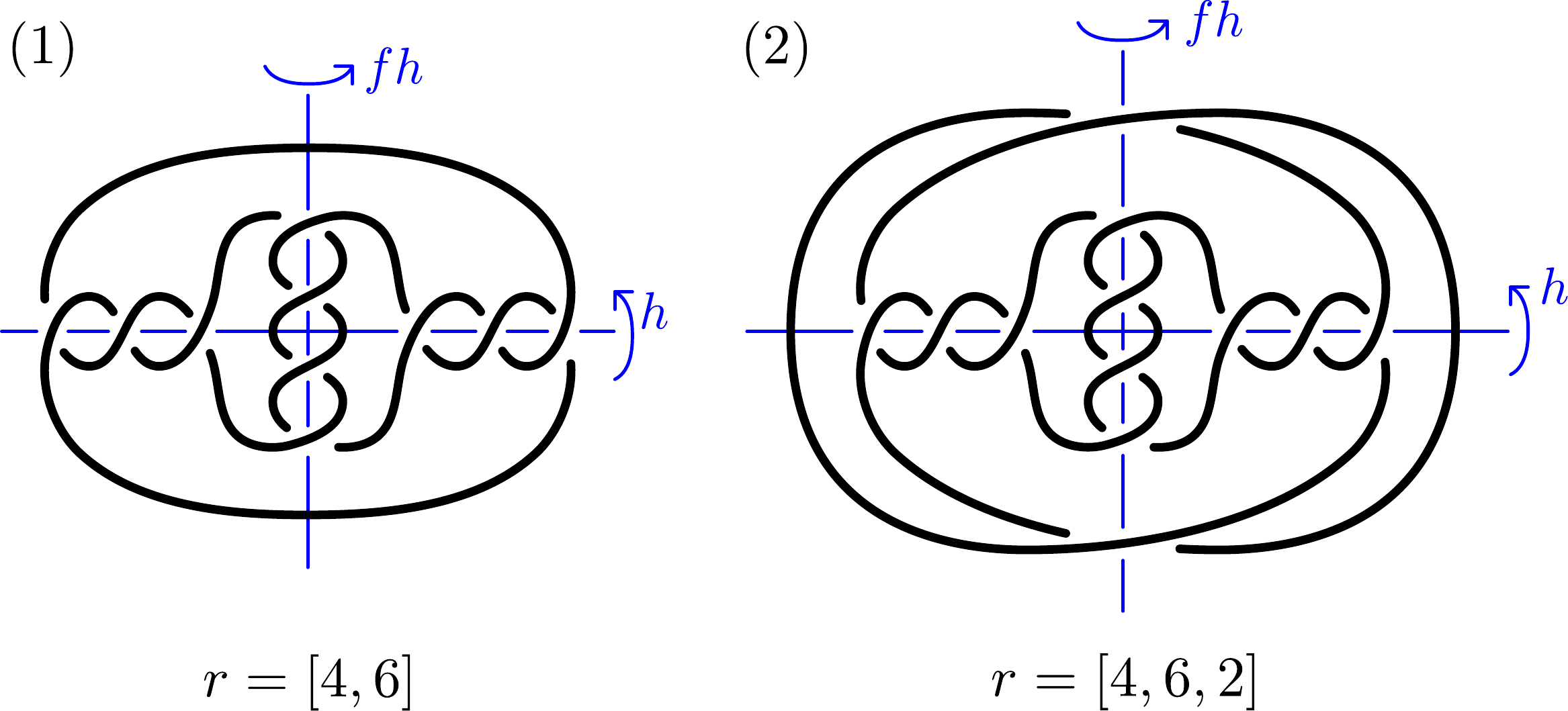}
	\caption{The action of the characteristic subgroup $(\ZZ_2)^2<\Sym^+(S^3,K(r))$
	for the knot case (1) and for the two component link case (2). In both cases, $f$
	is the $\pi$-rotation about the axis which intersects the projection plane
	perpendicularly at the central point (the intersection of the axes of $h$ and $fh$).}
	\label{Z2+Z2-symmetry}
\end{figure}

By using Lemma \ref{lem:conti-fract-expansion1},
we see that every $2$-bridge link $K(r)$ admits a $(\ZZ_2)^2$-action
generated by the two involutions $f$ and $h$
illustrated in Figure \ref{Z2+Z2-symmetry}(1) or (2)
according to whether $K(r)$ is a knot or a $2$-component link.
If $K(r)$ is hyperbolic, 
then the $(\ZZ_2)^2$-action projects faithfully onto  
the characteristic subgroup $(\ZZ_2)^2$ of
$\Sym^+(S^3,K(r))\cong\Isom^+(S^3-K(r))$.
(This can be seen either by appealing to the result of Borel
(see \cite{Conner-Raymond})
that a finite group action on an aspherical manifold $M$ with
centerless fundamental group projects injectively into $\Out(\pi_1(M))$, 
or by looking at the action on the canonical decomposition $\DD$
of $S^3-K(r)$.)
Moreover, any two such $(\ZZ_2)^2$-actions on $(S^3,K(r))$
are conjugate to each other,
because
(a) the restriction of any such $(\ZZ_2)^2$-action
to $E(K(r))$ is conjugate to the restriction of the action of 
the characteristic subgroup $(\ZZ_2)^2<\Isom^+(S^3-K(r))$
by virtue of the orbifold theorem, and because
(b) the restriction of any such $(\ZZ_2)^2$-action
to $N(K(r))$ is determined by its restriction to 
$\partial N(K(r))=\partial E(K(r))$.
Thus we obtain the following classification of the strong inversions
of $K(r)$, up to strong equivalence, 
whose image in $\Sym^+(S^3,K(r))$ is contained in 
the characteristic subgroup $(\ZZ_2)^2$.

\begin{enumerate}
\item
Suppose $K(r)$ is a hyperbolic $2$-bridge knot.
Then the involutions $h_+:=h$ and $h_-:=fh$ 
in Figure \ref{Z2+Z2-symmetry}(1)
are the only strong inversions of $K(r)$
which projects to an element of 
the characteristic subgroup $(\ZZ_2)^2<\Sym^+(S^3,K(r))$.
$\Fix(h_+)$ consists of two tunnels, one of which is 
the upper tunnel $\tau_+$.
We denote the other tunnel by $\tau_+^c$, and 
call the meridian pair of $\pi_1(E(K(r)))\cong G(r)$ represented by $\tau_+^c$
the {\it long upper meridian pair} of $G(r)$.
The tunnel $\tau_-^c$ and the {\it long lower meridian pair} of $G(r)$
are defined similarly.
\item
Suppose $K(r)$ is a two-component hyperbolic $2$-bridge link.
Then the involution $h$ illustrated in Figure \ref{Z2+Z2-symmetry}(2)
is the unique strong inversion of $K(r)$
which projects to an element of the characteristic subgroup
$(\ZZ_2)^2<\Sym^+(S^3,K(r))$.
$\Fix(h)$ consists of four tunnels, two of which
are $\tau_+$ and $\tau_-$.
We denote the remaining tunnels by $\tau_L$ and $\tau_R$,
and call the meridian pair of $G(r)$ represented by 
one of the two tunnels
an {\it intermediate meridian pair} of $G(r)$.
It should be noted that the two intermediate meridian pairs are 
equivalent modulo the automorphism $f_*$ of $G(r)$
induced by the involution $f$ in Figure \ref{Z2+Z2-symmetry}(2), 
because the two additional tunnels are mapped to each other by $f$.
\end{enumerate}

\begin{figure}
	\centering
	\includegraphics[width=10cm]{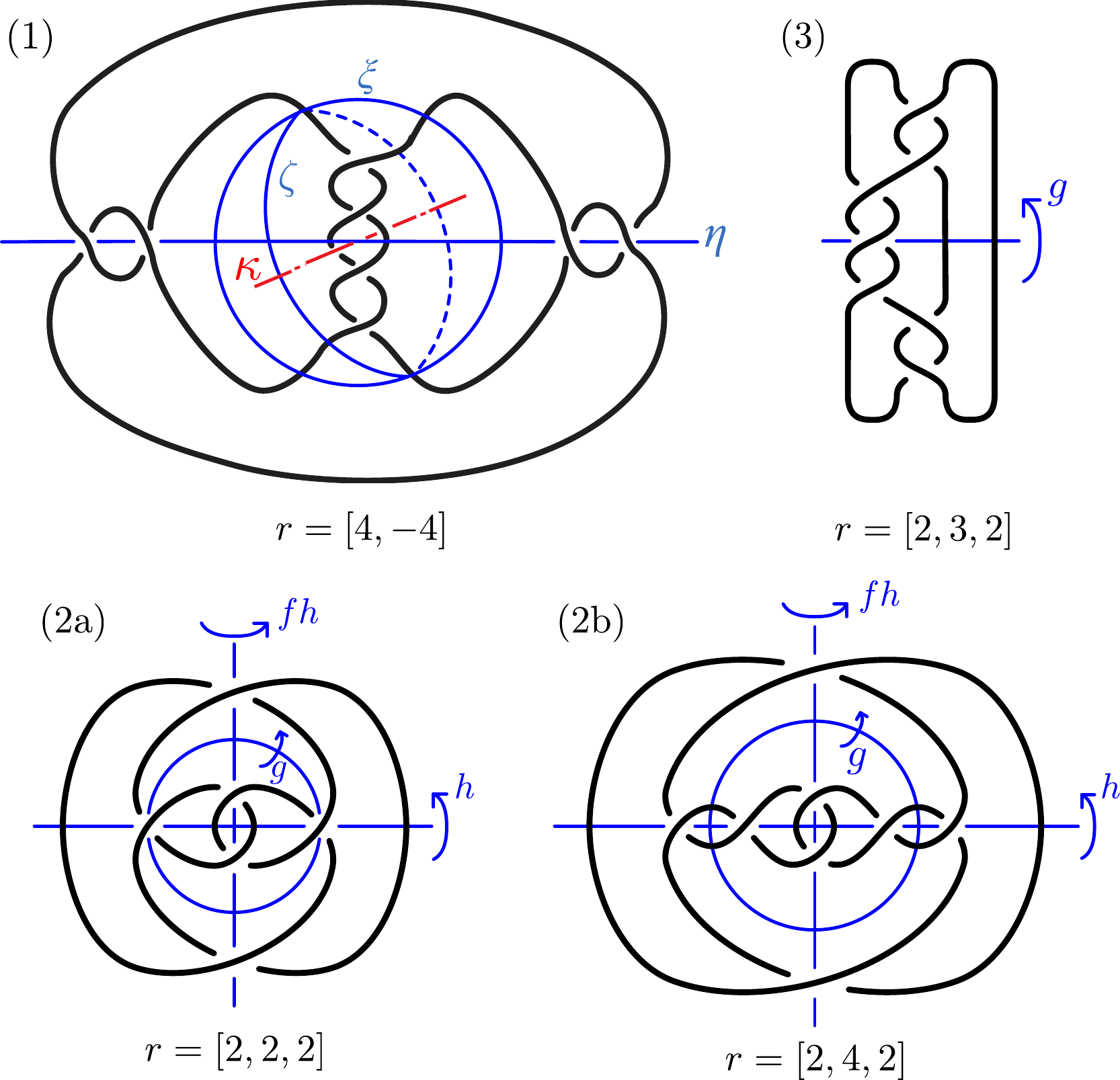}
\caption
{Additional symmetry of $K(r)$. 
(1) $p: \mbox{odd}$, $q^2\equiv 1 \pmod{p}$.	
Then $\Sym^+(S^3,K(r)) \cong \langle g, h \svert g^4, h^2, (gh)^2\rangle$,
where 
$g=\mbox{($\pi/2$-rotation about $\kappa$)} \circ \mbox{($\pi$-rotation about $\xi$)}$,
$h=\mbox{($\pi$-rotation around $\eta$)}$, and
$gh=\mbox{($\pi$-rotation around $\zeta$)}$.
In this case, $g^ih$ ($i=1,3$) are the extra strong inversions.
(2a) $p: \mbox{even}$, $q^2\equiv 1 \pmod{2p}$, $b_{n/2}: \mbox{odd}$.
In this case, $gh$ is the extra strong inversion.
(2b) $p: \mbox{even}$, $q^2\equiv 1 \pmod{2p}$, $b_{n/2}: \mbox{even}$.
In this case, $g$ is the extra strong inversion.
(3)  $p$: even, $q^2\equiv p+1 \pmod{2p}$.
In this case, there are no extra strong inversions.
}
\label{extra-symmetry}
\end{figure}

If $q^2\not\equiv 1 \pmod{p}$,
then $\Isom^+(S^3-K(r))\cong (\ZZ_2)^2$ by Proposition \ref{prop:isometry-group}, and so
the above $(\ZZ_2)^2$-action projects faithfully onto
the full group $\Isom^+(S^3-K(r))$.
Thus the list of strong inversions of $K(r)$ in the above
give all strong inversions of $K(r)$.
Hence, by Proposition \ref{prop:strong-involution_1}, 
the lists of meridian pairs in the above
give all possible parabolic generating pairs of $G(r)$.

If $q^2\equiv 1 \pmod{p}$,
then $\Isom^+(S^3-K(r))$ is a $\ZZ_2$-extension of the 
characteristic subgroup $(\ZZ_2)^2$,
and $(S^3,K(r))$ admits extra symmetries,
which interchange $\tau_+$ and $\tau_-$.

\begin{enumerate}
\item
Suppose $K(r)$ is a hyperbolic $2$-bridge knot (i.e., $p$ is odd) and $q^2\equiv 1 \pmod{p}$.
Then $n$ is even and $b_i=-b_{n+1-i}$ ($1\le i\le n$) by Lemma \ref{lem:conti-fract-expansion1}(1).
Thus $(S^3,K(r))$ admits an extra symmetry $g$ 
illustrated in Figure \ref{extra-symmetry}(1).
Note that $g$ and $h$ defined in the caption generate the oder $8$ dihedral group 
$D_4=\langle g, h \svert g^4, h^2, hgh^{-1}=g^{-1}\rangle$
which acts faithfully on $(S^3,K(r))$.
As in the preceding arguments,
we see that the $D_4$-action projects faithfully onto 
the whole group $\Sym^+(S^3,K(r))$ 
and that any such $D_4$-action on $\Sym^+(S^3,K(r))$
is conjugate to the $D_4$-action in Figure \ref{extra-symmetry}(1).
Hence the involutions $g^ih$ ($0\le i\le 3$) are 
the only strong inversions of $K(r)$
(cf. \cite[Section 3]{Sakuma1}).
Thus, in addition to $h_+=h$ and $h_-=g^2h$,
we have two extra strong inversions $gh$ and $g^3h$,
and each of them determines two tunnels for $K(r)$.
We call the meridian pairs of $G(r)$ determined by such an extra tunnel
an {\it extra meridian pair}  of $G(r)$.
The four extra meridian pairs are divided into two classes, 
where each class consists of two extra meridian pairs
which are equivalent modulo the  
automorphism $g_*$ of $G(r)$ induced by $g$,
because $gh$ and $g^3h=g(gh)g^{-1}$ are conjugate in $D_4$.

\item
Suppose $K(r)$ is a two-component hyperbolic $2$-bridge link  (i.e., $p$ is even) 
and $q^2\equiv 1 \pmod{2p}$.
Then $n$ is odd and $b_i=b_{n+1-i}$ ($1\le i\le n$) by 
Lemma \ref{lem:conti-fract-expansion1}(2).
Thus $(S^3,K(r))$ admits a $(\ZZ_2)^3$-action 
generated by three involutions $f$, $h$ and $g$,
as shown in Figure \ref{extra-symmetry}(2a) and (2b)
according to whether $b_{n/2}$ is odd or even.
Then 
$h':=gh$ or $h':=g$ is a strong inversion of $K(r)$ accordingly,
and this is the only extra strong inversion of $K(r)$.
We call the meridian pair represented by one of the four arc components
of $\Fix(h')\cap E(K(r))$ an {\it extra meridian pair}.
The four extra meridian pairs are divided into two classes, 
where each class consists of two extra meridian pairs which are 
equivalent modulo the  
automorphism $f_*$ of $G(r)$ induced by $f$.

\item
Suppose $K(r)$ has two components and $q^2\equiv p+1 \pmod{2p}$.
In this case, it is not easy to draw a link diagram
in which one can see the whole symmetry.
(A simple conceptual way to understand the whole symmetry is to 
use the decomposition of $(S^3,K(r))$ into 
two rational tangles, as described in Section \ref{sec:extra-meridian-pair}).
However, it is easy to visualise a single additional symmetry,
as described below.
By \cite[Lemma II.3.3]{Sakuma-Weeks}
(cf. Lemma \ref{lem:conti-fract-expansion2}(1) 
in Section \ref{sec:extra-meridian-pair}),
the continued fraction expansion of $r=q/p$ 
consisting of positive integers is symmetric,
and $(S^3,K(r))$ admits an extra symmetry $g$ which interchanges $\tau_+$ and $\tau_-$,
as shown in Figure \ref{extra-symmetry}(3).
The extra involution $g$ preserves each of the two components of $K(r)$,
acts on one component preserving orientation and on the other component
reversing orientation.
Since $\Sym^+(S^3,K(r))\cong D_4$ is generated by 
the characteristic subgroup $(\ZZ_2)^2$
and the extra element $g$,
we see that $K(r)$ does not have an extra strong inversion.
\end{enumerate}

\medskip
By the above arguments,
we obtain the following propositions
(cf. \cite[Corollary 2.2]{Lee-Sakuma_2016}).

\begin{proposition}
\label{prop:strong-invertion_1}
For a hyperbolic $2$-bridge knot $K(r)$ with $r=q/p$ 
($p$:odd, $q\not\equiv \pm1\pmod{p}$),
the following hold.

{\rm (1)}
Suppose $q^2\not\equiv 1\pmod p$.
Then any parabolic generating pair of $G(r)$
is equivalent to the upper, lower,
long upper or long lower meridian pair.

{\rm (2)}
Suppose $q^2\equiv 1\pmod p$.
Then any parabolic generating pair of $G(r)$
is equivalent to the upper, lower,
long upper, long lower meridian pair
or one of the four extra 
meridian pairs.
Moreover, the four extra 
meridian pairs are divided into two classes
up to automorphisms of $G(r)$.
\end{proposition}

\begin{proposition}
\label{prop:strong-invertion_2}
For a hyperbolic $2$-component $2$-bridge link $K(r)$ 
with $r=q/p$ 
($p$:even, $q\not\equiv \pm1\pmod{p}$),
the following hold.

{\rm (1)}
Suppose $q^2\not\equiv 1\pmod{2p}$.
Then any parabolic generating pair of $G(r)$
is equivalent to the upper or lower meridian pair,
or one of the two intermediate meridian pairs.
Moreover, the two intermediate meridian pairs are equivalent
up to automorphisms of $G(r)$. 

{\rm (2)}
Suppose $q^2\equiv 1\pmod{2p}$.
Then any parabolic generating pair of $G(r)$
is equivalent to the upper or lower meridian pair,
one of the two intermediate meridian pairs, or 
one of the four extra meridian pairs.
Moreover, the two intermediate meridian pairs
are equivalent 
up to automorphisms of $G(r)$,
and the four extra meridian pairs
are divided into two classes
up to automorphisms of $G(r)$.
\end{proposition}

By Propositions \ref{prop:strong-invertion_1} 
and \ref{prop:strong-invertion_2},
the proof of the assertion in Theorem \ref{main-theorem2}(1)
that each hyperbolic $2$-bridge link group $G(r)$
admits at most two parabolic generating pairs
(i.e., the upper/lower meridian pairs)
is reduced to the proof of the fact that
none of the long upper/lower meridian pairs,
intermediate meridian pairs,
and the extra meridian pairs can
generate the hyperbolic $2$-bridge link group $G(r)$.
The next two sections are devoted to the proof of this fact.

At the end of this section, 
we prove the following proposition,
which, together with the above, completes the proof of Theorem \ref{main-theorem2}(1). 

\begin{proposition}
\label{prop:upper-lower-difference}
For each hyperbolic $2$-bridge link group $G(r)$,
the upper meridian pair and the lower meridian pair are not equivalent.
\end{proposition}

\begin{proof}
Suppose the upper and lower meridian pairs 
of a hyperbolic $2$-bridge link group $G(r)$
are equivalent.
Then the upper tunnel $\tau_+$ and the lower tunnel $\tau_-$
are properly homotopic in $E(K(r))$.
This contradicts \cite[Example (3.4)]{Morimoto-Sakuma},
which implies that $\tau_+$ and $\tau_-$ are not properly homotopic.
(Though it is only claimed that they are not isotopic,
the proof actually shows that they are not properly homotopic.) 
\end{proof}

\section{Long upper/lower meridian pairs and intermediate peridian pairs}
\label{sec:long-meridian-pair}

For a link $K$ in $S^3$, let $M(K)$ be the double branched covering of $S^3$ branched over $K$.
Then its fundamental group is intimately related with the
{\it $\pi$-orbifold group} $\Piorb(K)$ of $K$,
which is defined as the quotient of the link group
$\pi_1(E(K))$ by the normal closure of the squares of meridians
(see \cite{Boileau-Zimmermann}).
In fact, $\Piorb(K)$ is the semidirect product
$\pi_1(M(K)) \rtimes \ZZ_2$,
where the action of $\ZZ_2$ on $\pi_1(M(K))$ is
given by the action of the covering transformation group.
If $K$ is a $2$-bridge link $K(r)$ with $r=q/p$,
then $\Piorb(r):=\Piorb(K(r))$ is isomorphic to the semidirect product
\[
\pi_1(M(K(r))) \rtimes \ZZ_2 \cong
H_1(M(K(r))) \rtimes \ZZ_2
\cong
\ZZ_p \rtimes \ZZ_2
\cong D_p,
\]
where $D_p$ is the dihedral group of order $2p$.

For a meridian pair $\{m_1,m_2\}$ in $\pi_1(E(K))$,
let $\omega(m_1,m_2)$ be the element of $\Piorb(K)$
represented by the product $m_1m_2\in \pi_1(E(K))$.
The following simple observation is a key tool
for the proof of Theorem \ref{main-theorem2}(1).

\begin{lemma}
\label{lem:invariant-omega}
(1) The element $\omega(m_1,m_2)\in \Piorb(K)$, up to inversion and conjugation,
is uniquely determined by the equivalence class of $\{m_1,m_2\}$.
Namely, if two meridian pairs $\{m_1,m_2\}$ and $\{m_1',m_2'\}$ of 
$\pi_1(E(K))$
are equivalent, then $\omega(m_1',m_2')$ is conjugate to 
$\omega(m_1,m_2)$ or its inverse in $\Piorb(K)$.

(2)
If $\{m_1,m_2\}$ is a generating meridian pair
of a $2$-bridge link group $G(r)$ with $r=q/p$,
then $\omega(m_1,m_2)$ is a generator
of the the index $2$ subgroup $H_1(M(K(r)))\cong \ZZ_p$ of
$\Piorb(r)\cong H_1(M(K(r))) \rtimes \ZZ_2 \cong D_p$.
\end{lemma}

\begin{proof}
(1) follows from the fact that $m_1^2=m_2^2=1$ in $\Piorb(K)$
and the definition of the equivalence of meridian pairs.

(2)
Suppose that $\{m_1,m_2\}$ is a generating meridian pair
of a $2$-bridge link group $G(r)$ with $r=q/p$.
Then there is an epimorphism from 
the infinite dihedral group 
$D_{\infty}=\langle m_1, m_2 \svert m_1^2, \ m_2^2\rangle$
generated by the symbols $m_1$ and $m_2$
onto $\Piorb(r)\cong D_p$
which maps the elements $m_1$ and $m_2$ of $D_{\infty}$
to the (images of the) meridians $m_1$ and $m_2$ in $\Piorb(r)$, respectively.
Since $D_{\infty}$ is isomorphic to the semi-direct product
$\langle m_1m_2\rangle \rtimes \ZZ_2$,
it follows that $\Piorb(r)\cong D_p$ is the quotient of
$D_{\infty}$
by the infinite cyclic normal subgroup
$\langle (m_1m_2)^p\rangle$.
Hence $\omega(m_1,m_2)=m_1m_2$ generates 
$H_1(M(K(r)))\cong \langle m_1m_2 \svert p(m_1m_2)=0 \rangle$.
\end{proof}

Note that the long upper/lower meridian pairs are defined for all 
$2$-bridge knots and that
the intermediate meridian pairs are defined for all 
$2$-component $2$-bridge links.
The following proposition together with the above lemma
shows that the long upper/lower meridian pairs 
and the intermediate meridian pairs are not generating pairs
of hyperbolic $2$-bridge link groups.

\begin{proposition}
\label{prop:long-meridian-pair}
(1) Let $K(r)$ with $r=q/p$ ($p$: odd) be a $2$-bridge knot,
and let $\{m_1,m_2\}$ be the long upper (or lower) meridian pair
of the knot group $G(r)$.
Then $\omega(m_1,m_2)=0$ in $H_1(M(K(r)))\cong \ZZ_p$.
In particular, if $K(r)$ is a nontrivial knot
(i.e. $p\ge 3$), then $\omega(m_1,m_2)$ is not a generator of $H_1(M(K(r)))$.

(2) Let $K(r)$ with $r=q/p$ ($p$ even) be a nontrivial 
$2$-component $2$-bridge link,
and let $\{m_1,m_2\}$ be an intermediate meridian pair
of the link group $G(r)$.
Then $2\omega(m_1,m_2)=0$ in $H_1(M(K(r)))\cong \ZZ_p$.
In particular, if $K(r)$ is not a Hopf link
(i.e. $p\ge 4$), 
then $\omega(m_1,m_2)$ is not a generator of $H_1(M(K(r)))$.
\end{proposition}

\begin{figure}
	\centering
	\includegraphics[width=9cm]{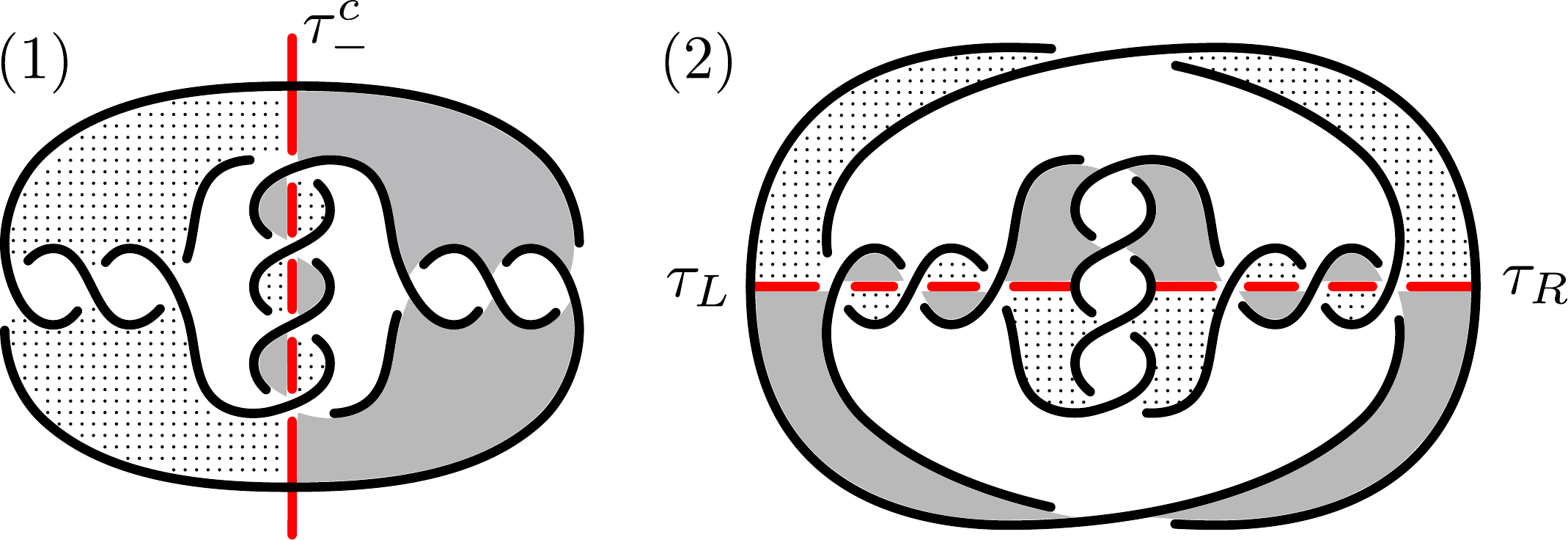}
	\caption{Checkerboard surfaces associated with the diagrams
	in Figure \ref{Z2+Z2-symmetry} containing $\tau_-^c$ in (1)
	and $\tau_L\cup\tau_R$ in (2), respectively, as separating arc systems}
	\label{checkerboard-surface}
\end{figure}

\begin{proof}
(1) We prove the assertion for the long lower meridian pair.
(The assertion for the long upper meridian pair is a consequence of this,
because there is an isomorphism $G(q/p)\cong G(q'/p)$,
where $qq'\equiv 1 \pmod{p}$,
which maps the long upper meridian pair of $G(q/p)$
to the long lower meridian pair of $G(q'/p)$
(cf. \cite[Proposition 2.1]{AOPSY})).
Recall that the long lower meridian pair $\{m_1,m_2\}$ is represented by
the tunnel $\tau_-^c\subset \Fix(h_-)$.
Then $\tau_-^c$ is properly embedded in the checker-board surface, $F$,
illustrated in Figure \ref{checkerboard-surface}(1),
associated with the link diagram in Figure \ref{Z2+Z2-symmetry}(1),
and moreover, $\tau_-^c$ is separating in $F$.
Let $M_F$ be the $3$-manifold obtained by cutting $S^3$ along $F$,
and let $\sigma$ be the involution on $\partial M_F$
such that $\Fix(\sigma)$ is the copy of $L$ in $\partial M_F$
and that 
$(M_F,\partial M_F)/\sigma \cong (S^3, F)$.
(Here the symbol $/\sigma$ means to identify $x$ with $\sigma(x)$
for all $x\in\partial M_F$.)
The double branched covering $M(K(r))$ of $S^3$ branched over $K(r)$
is obtained from two copies $M_F^{(0)}$ and $M_F^{(1)}$ of $M_F$
by gluing their boundaries through the homeomorphism
$\partial M_F^{(0)}\to \partial M_F^{(1)}$ induced by $\sigma$,
i.e. the homeomorphism that maps the copy in $M_F^{(0)}$ of a point $x\in \partial M_F$ 
to its copy in $M_F^{(1)}$ of the point $\sigma(x)\in \partial M_F$.
Now, let $\tilde\tau_-^c$ be the simple loop in $\partial M_F$ 
obtained as the inverse image of the proper arc $\tau_-^c\subset F$ in $\partial M_F$
under the projection $\partial M_F \to \partial M_F/\sigma = F$.
By using the fact that
the long lower meridian pair $\{m_1,m_2\}$ is represented by $\tau_-^c$,
we can see that the element $\omega(m_1,m_2)\in H_1(M(K(r))) < \Piorb(r)$ 
is represented by the simple loop $\tilde\tau_-^c$.
To be precise, $\omega(m_1,m_2)$ is represented by
the copy of the loop $\tilde\tau_-^c$ in $\partial M_F^{(0)}\subset M(K(r))$
with a suitable orientation.
However, the loop $\tilde\tau_-^c$ is separating in $\partial M_F$,
because it is the inverse image of the separating arc $\tau_-^c$ of $F$.
Thus $\tilde\tau_-^c$ is null-homologous in $M_F$, and
hence 
$\omega(m_1,m_2)=[\tilde\tau_-^c]=0 \in H_1(M(K(r)))$, as desired.

(2) 
We prove the assertion for the intermediate meridian pair $\{m_1,m_2\}$ 
represented by the tunnel $\tau_L \subset \Fix(h)$.
(The assertion for that represented by $\tau_R$ is a consequence of this
by Proposition \ref{prop:strong-invertion_2}(1).)
Observe that the checkerboard surface, $F$, for $K(r)$ in 
Figure \ref{checkerboard-surface}(2)
contains $\tau_L\cup \tau_R$ as a separating arc system.
Observe also that the involution $f$ 
of $(S^3,K(r))$
introduced in Figure \ref{Z2+Z2-symmetry}(2)
preserves the surface $F$ and interchanges $\tau_L$ and $\tau_R$.
Let $M_F$ and $\sigma$ be as in the proof of (1), and recall that
the double branched covering $M_2(K(r))$ is obtained
by gluing two copies $M_F^{(0)}$ and $M_F^{(1)}$ of $M_F$
through the homeomorphism
$\partial M_F^{(0)}\to \partial M_F^{(1)}$ induced by 
the involution $\sigma$ on $\partial M_F$.
Now, let $\tilde\tau_L$ and $\tilde\tau_R$ be the simple loops
in $\partial M_F$ obtained as the 
inverse image of the arcs $\tau_L$ and $\tau_R$ in $F$, respectively,
under the projection $\partial M_F \to \partial M_F/\sigma = F$
($i=1,2$).
As in (1), the element $\omega(m_1,m_2)\in H_1(M(K(r))) < \Piorb(r)$ 
is represented by the simple loop 
$\tilde\tau_L\subset \partial M_F\cong \partial M_F^{(0)}\subset M_2(K(r))$
with a suitable orientation.
We endow $\tilde\tau_L$ with this orientation.

Now let $\hat f$ be the orientation-preserving involution on $M_F$ induced by 
the involution $f$. 
Then $\tilde \tau_R=\hat f(\tilde \tau_L)$, 
and we orient $\tilde \tau_R$ as the image  by $\hat f$
of the oriented loop $\tilde \tau_L$.
On the other hand,
since $\tau_L\cup \tau_R$ is a separating arc system in $F$,
$\tilde\tau_L\cup \tilde\tau_R$ is a separating loop system
in $\partial M_F$
(where we forget the orientation),
and $\hat f$ interchanges the two components of 
$\partial M_F-(\tilde\tau_L\cup \tilde\tau_R)$.
Since (the restriction to $\partial M_F$ of) $\hat f$ is
orientation-preserving,
this implies that the cycle 
$\tilde\tau_L-\hat f(\tilde\tau_L)=\tilde\tau_L-\tilde \tau_R$
is null homologous in $\partial M_F$.
Hence we have $[\tilde\tau_L]=[\hat f(\tilde\tau_L)]$ in $H_1(M(K(r)))$.

Let $\tilde f$ be the lift of $f$ to $M_2(K(r))=M_F^{(0)}\cup M_F^{(1)}$
that is obtained  by gluing the copies of the involution $\hat f$ on $M_F^{(i)}$ ($i=1,2$).
Note that
the double branched covering $M(K(r))$ is the union of two solid tori,
whose cores are the circles, $\tilde \tau_+$ and $\tilde\tau_-$,
obtained as the inverse images of
the upper tunnel $\tau_+$ and the lower tunnel $\tau_-$, respectively
(cf. the second paragraph of Section \ref{sec:extra-meridian-pair}).
Since $\Fix(f)$ intersects $\tau_{\pm}$ transversely in a single point
and since $\tau_{\pm}\subset M_F$, we see that
$\Fix(\tilde f)$ intersects each of 
the core circle $\tilde\tau_{\pm}$ in two points, and so
$\tilde f$ acts on the circles $\tilde\tau_{\pm}$ as an orientation-reversing involution.
Since each of $[\tilde\tau_{\pm}]$ is a generator of $H_1(M(K(r)))$,
this implies that $\tilde f_*$ acts on $H_1(M(K(r)))$
as multiplication by $-1$.
Hence, we have
\[
[\tilde\tau_L]=[\hat f(\tilde\tau_L)]
=\tilde f_*([\tilde\tau_L])=-[\tilde\tau_L] \in H_1(M(K(r)))
\]
Thus we have $2\omega(m_1,m_2)=2[\tilde\tau_L]=0$ in $H_1(M(K(r)))$,
as desired.
\end{proof}

\section{Extra meridian pairs}
\label{sec:extra-meridian-pair}

In this section, we prove the following proposition,
which, together with Lemma \ref{lem:invariant-omega}, implies that
the extra meridian pairs are not generating pairs of hyperbolic $2$-bridge link groups.

\begin{proposition}
\label{prop:extra-meridian-pair}
Let $\{m_1,m_2\}$ be an extra meridian pair
of a hyperbolic $2$-bridge link group $G(r)$.
Then $\omega(m_1,m_2)\in H_1(M(K(r)))$
is not a generator of $H_1(M(K(r)))$.
\end{proposition}

To this end,
we regard $(S^3,K(r))$ as the union of two rational tangles
$\rtangle{\infty}$ and $\rtangle{r}$
of slopes $\infty$ and $r$,
as in \cite[Section 2]{AOPSY}.
Here the common boundary
$\partial \rtangle{\infty}=\partial \rtangle{r}$
is identified with the Conway sphere $\Conways:=(\RR^2,\ZZ^2)/\RotG$,
where $\RotG$ is the group of isometries
of the Euclidean plane $\RR^2$
generated by the $\pi$-rotations around
the points in the lattice $\ZZ^2$.
For each rational number $s\in\QQQ:=\QQ\cup\{\infty\}$,
a line of slope $s$ in $\RR^2-\ZZ^2$
projects to an essential simple loop,
denoted by $\alpha_s$, in the $4$-times punctured sphere $\PConway:=\Conway-\PP$.
Similarly,
a line of slope $s$ in $\RR^2$ passing through a point $\ZZ^2$
determines an essential simple proper arc in $\PConway$.
For each $s$, there are exactly two essential simple proper arcs in $\PConway$
obtained in this way, 
and the union of the two arcs is denoted by $\delta_s$.
The rational number $s$ is called the {\it slope} of $\alpha_s$ and $\delta_s$.
By the definition of the rational tangles,
the loops $\alpha_{\infty}$ and $\alpha_r$ bound discs
in $\threeball-t(\infty)$ and $\threeball-t(r)$, respectively.

The double branched covering $M(K(r))$ of $(S^3,K(r))$
is the union of the solid tori $V_{\infty}$ and $V_r$
that are obtained as the double branched coverings
of $\rtangle{\infty}$ and $\rtangle{r}$, respectively.
Let $\tilde\alpha_0$ and $\tilde\alpha_{\infty}$ be lifts in $\partial V_{\infty}$ of
the simple loops $\alpha_0$ and $\alpha_{\infty}$, respectively.
(Here a lift means a connected component of the inverse image.)
Then $\tilde\alpha_0$ and $\tilde\alpha_{\infty}$
form the meridian and the longitude of $V_{\infty}$.
Similarly a lift $\tilde\alpha_r$ of $\alpha_r$ in $\partial V_r$ is a meridian of $V_r$.
Thus the loops
$\tilde\alpha_{\infty}$ and $\tilde\alpha_{r}$ represent the trivial elements of
$H_1(V_{\infty})$ and $H_1(V_r)$, respectively.
Since
$[\tilde\alpha_{r}]=p[\tilde\alpha_0]+q[\tilde\alpha_{\infty}]$ in
$H_1(\partial V_{\infty})$, where $r=q/p$, 
we have
\begin{align*}
H_1(M(K(r))
\cong 
\langle [\tilde\alpha_{0}], [\tilde\alpha_{\infty}] \ | \
[\tilde\alpha_{\infty}], \ [\tilde\alpha_r]\rangle
\cong 
\langle \tilde\alpha_{0} \ | \
p[\tilde\alpha_0]\rangle
\cong
\ZZ_p.
\end{align*}

\medskip

Now recall the following well-known facts
(cf. \cite[Lemma II.3.3]{Sakuma-Weeks}).

\begin{lemma}
\label{lem:conti-fract-expansion2}
For a rational number $r=q/p$ with $0<q\leq p/2$,
consider the continued fraction expansion
$r=[a_1,a_2, \cdots,a_{n}]$ into positive integers $a_i$
such that $a_1\ge 2$ and $a_n\ge 2$.
Then the following hold.

(1) The following conditions are equivalent.
\begin{enumerate}
\item[(a)]
$p$ is even and $q^2\equiv 1 \pmod{2p}$.
\item[(b)]
$n$ is odd, $a_{(n+1)/2}$ is even, and
$(a_1,\cdots a_n)$ is symmetric.
\end{enumerate}

(2) The following conditions are equivalent.
\begin{enumerate}
\item[(a)]
Either (i) $p$ is odd and $q^2\equiv 1 \pmod{p}$,
or (ii) $p$ is even and $q^2\equiv p+1 \pmod{2p}$
\item[(b)]
$n$ is odd, $a_{(n+1)/2}$ is odd, and
$(a_1,\cdots a_n)$ is symmetric.
\end{enumerate}
\end{lemma}

\begin{figure}
	\centering
	\includegraphics[width=9cm]{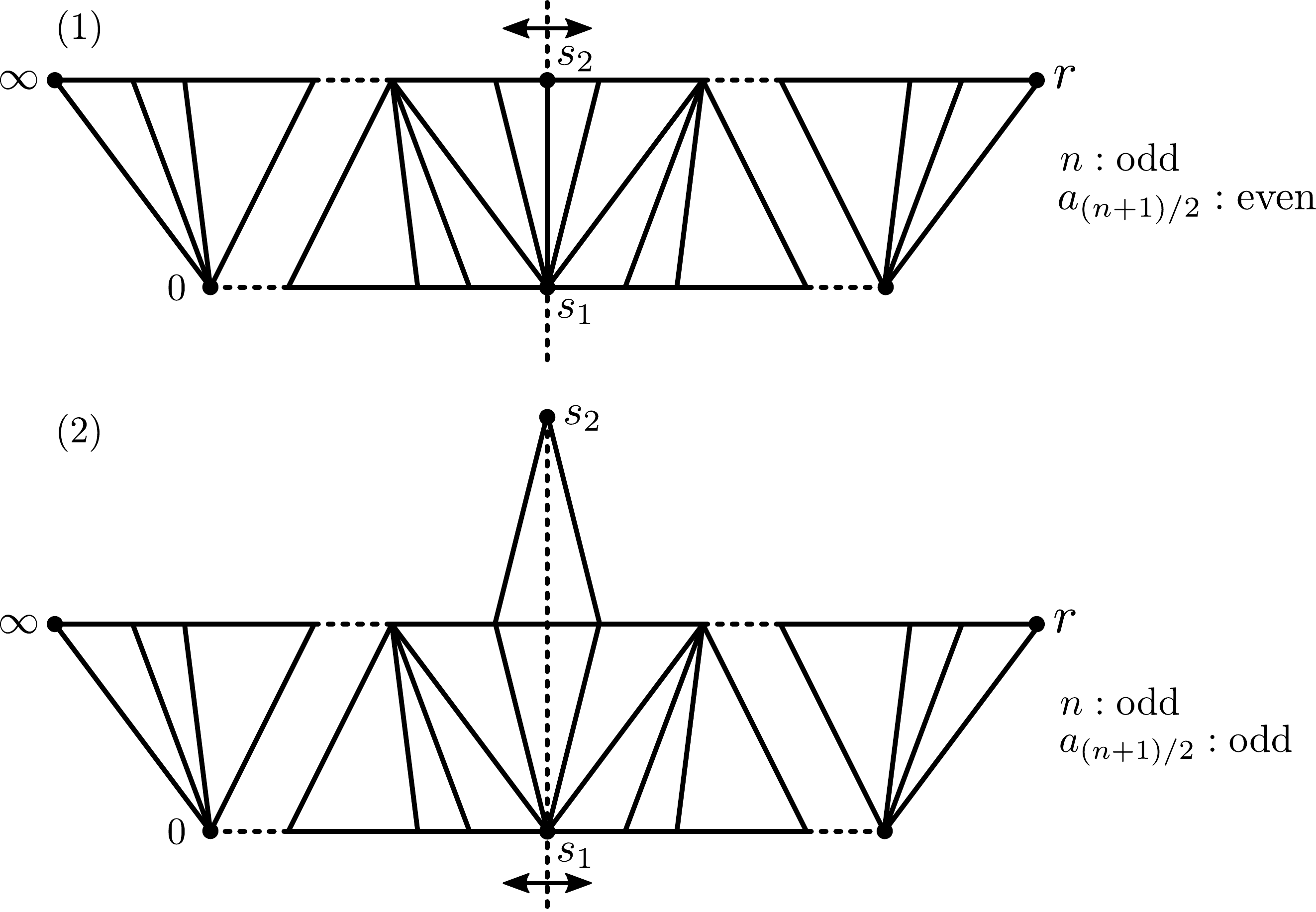}
	\caption{If $r=q/p$ satisfies one of the conditions
	in Lemma \ref{lem:conti-fract-expansion2},
	then there is an orientation-reversing involution of the Farey tessellation 
	$\Farey$ which interchanges $\infty$ and $r$.}
	\label{Farey-Symmetry}
\end{figure}

The symmetry of the continued fraction expansion in the above lemma 
is realised by the symmetry of the Farey tessellation, $\Farey$,
as illustrated by 
Figure \ref{Farey-Symmetry}.
Recall that the
{\it Farey tessellation} is
the tessellation of the upper half
space $\HH^2$ by ideal triangles that are obtained
from the ideal triangle with the ideal vertices $0, 1,\infty \in \QQQ$ by repeated reflection in the edges.
Then $\QQQ$ is identified with the set of the vertices of $\Farey$.
The automorphism group of the Farey tessellation is identified with
$\PGL(2,\ZZ)$ as follows.
For a matrix 
$
A=\begin{pmatrix}
a & b\\
c & d
\end{pmatrix}
$,
consider its action on $\RR^2$ 
by the left multiplication 
$
\begin{pmatrix}
x\\ y
\end{pmatrix}
\mapsto
A
\begin{pmatrix}
x\\ y
\end{pmatrix}
$.
Then $A$ maps a line of a slope $s\in \QQQ$
to a line of slope $s'=(c+ds)/(a+bs)$.
The correspondence $s\mapsto s'$ gives a bijection of the Farey vertices,
which extends to an automorphism of the Farey tessellation.
(When $A\in \SL(2,\ZZ)$, the transformation $z\mapsto (c+dz)/(a+bz)$ of $\bar\HH^2\cup\{\infty\}$
gives the desired automorphism.)
Conversely, every automorphism
of the Farey tessellation $\Farey$ is obtained in this way.

Now suppose $q^2\equiv 1 \pmod p$.
Then, by Lemma \ref{lem:conti-fract-expansion2}, there is an orientation-reversing 
involution of the Farey tessellation which interchanges $\infty$ with $r$
(see Figure \ref{Farey-Symmetry}).
Let $R$ be one of the two matrices in $\GL(2,\ZZ)$ that realises the involution.
Then the linear action of $R$ on $\RR^2$ is orientation-reversing 
and interchanges 
the $1$-dimensional vector subspace of slope $\infty$ 
with that of slope $r$.
Since $R$ normalizes the subgroup $\RotG$ in the affine transformation group of $\RR^2$,
it descends to an orientation-reversing involution on $\Conways$
which swaps $\delta_{\infty}$ with $\delta_r$
and $\alpha_{\infty}$ with $\alpha_r$.
Hence $R$ induces an orientation-preserving involution, $g_R$, of
$(S^3,K(r))$ which interchanges $(\Ball, t(\infty))$ with 
$(\Ball, t(r))$.
This gives one of the extra symmetries of $(S^3,K(r))$ described in Section \ref{sec:two-bridge-link-group}.
We note that the generators $f$ and $h$ of 
the $(\ZZ_2)^2$-action on $(S^3,K(r))$ of 
the characteristic subgroup come from 
the $\pi$-rotations of $\RR^2$ about the point $(1/2,0)$ and $(0,1/2)$, respectively.
The extra symmetry $g_R$ and the $(\ZZ_2)^2$-action
generate a finite group action on $(S^3,K(r))$
which gives a realisation of $\Sym^+(S^3,K(r))$.

\medskip

Case 1.
Suppose that the mutually equivalent conditions 
in Lemma \ref{lem:conti-fract-expansion2}(1) hold,
namely, $n=2n_0+1$ is odd, $a_{n_0+1}=2a_{n_0+1}'$ is even, and
$(a_1,\cdots a_n)$ is symmetric.
Then the reflection on $\Farey$ in the Farey edge, 
spanned by the vertices $s_1=q_1/p_1$ and $s_2=q_2/p_2$ 
in Figure \ref{Farey-Symmetry}(1),
interchanges $\infty$ and $r$.
Thus the matrix $\pm R$ is the conjugate of the matrix
$R_0=
\begin{pmatrix}
1 & 0\\
0 & -1
\end{pmatrix}$, 
which induces the reflection on $\Farey$
in the Farey edge spanned by $0/1$ and $1/0$,
by the the matrix 
$
A=
\begin{pmatrix}
p_1 & p_2\\
q_1 & q_2
\end{pmatrix}
$,
which maps the vectors
$
\begin{pmatrix}
1\\ 0
\end{pmatrix}
$
and 
$
\begin{pmatrix}
0\\ 1
\end{pmatrix}
$
to the vectors
$
\begin{pmatrix}
p_1\\ q_1
\end{pmatrix}
$
and 
$
\begin{pmatrix}
p_2\\ q_2
\end{pmatrix}
$,
respectively.
Hence we have
\[
\pm R= AR_0A^{-1}=
\begin{pmatrix}
p_1q_2+p_2q_1 & -2p_1p_2\\
2q_1q_2 & -p_1q_2-p_2q_1
\end{pmatrix}.
\]
Thus
\[
\begin{pmatrix}
p\\
q
\end{pmatrix}
=
\pm
R
\begin{pmatrix}
0\\
1
\end{pmatrix}
=
\pm \begin{pmatrix}
-2p_1p_2\\
-p_1q_2-p_2q_1
\end{pmatrix}.
\]
Hence we have $p=2p_1p_2$.

Observe that the involution on $\Conways$ induced by $R_0$
is the reflection in the circle that is the union of
the four arcs in $\delta_{0/1}$ and $\delta_{1/0}$.
Hence the restriction of the extra involution $g_R$ to $\Conways$
is the the reflection in the circle that is the union 
the four arcs in $\delta_{q_1/p_1}$ and $\delta_{q_2/p_2}$.
This confirms that $g_R$ is a strong inversion of the $2$-component link $K(r)$,
and the four extra meridian pairs 
that arise from the strong inversion $g_R$
are represented by 
the four arcs in $\delta_{q_1/p_1}$ and $\delta_{q_2/p_2}$, respectively
(cf. Figure \ref{extra-symmetry}(2a), (2b)).
Let $\{m_1,m_2\}$ be one of the four extra meridian pairs.
Then it is represented by an arc on $\PConway$ of slope $q_i/p_i$ for $i=1$ or $2$.
This implies that the element $\omega(m_1,m_2)\in H_1(M(K(r)))<\Piorb(r)$ is represented 
the simple loop $\alpha_{q_i/p_i}$ for $i=1$ or $2$.
Hence we have
\[
\omega(m_1,m_2) = [\alpha_{q_i/p_i}] 
=p_i[\tilde\alpha_0]+q_i[\tilde\alpha_{\infty}]
=p_i[\tilde\alpha_0]
\in
H_1(M(K(r))
\cong \langle \tilde\alpha_{0} \ | \
p[\tilde\alpha_0]\rangle.
\]
If $K(r)$ is hyperbolic, then $q \not\equiv \pm1 \pmod{p}$ and so
we can see that $1< p_i < p=2p_1p_2$
(see Figure \ref{Farey-Symmetry}(1)). 
So, $\omega(m_1,m_2)$ is not a generator of $H_1(M(K(r)))$.
Hence, none of the extra meridian pairs is a generating pair of $G(r)$.
This completes the proof of Proposition \ref{prop:extra-meridian-pair}
for the case where the condition (1) in Lemma \ref{lem:conti-fract-expansion2}
is satisfied.

\medskip

Case 2.
Suppose that the mutually equivalent conditions 
in Lemma \ref{lem:conti-fract-expansion2}(2) hold,
namely,
$n=2n_0+1$ is odd, $a_{n_0+1}=2a_{n_0+1}'+1$ is odd, and
$(a_1,\cdots a_n)$ is symmetric.
Then the reflection on $\Farey$ in the geodesic, 
joining the vertices $s_1=q_1/p_1$ and $s_2=q_2/p_2$
in Figure \ref{Farey-Symmetry}(2),
interchanges $\infty$ and $r$.
Thus the matrix $\pm R$ realising the symmetry is conjugate to the matrix 
$
R_0=
\begin{pmatrix}
0 & 1\\
1 & 0
\end{pmatrix}
$,
which induces the reflection on $\Farey$ in 
the geodesic joining the vertices $1/1$ and $-1/1$,
by the matrix 
$A$
that maps the vectors
$
\begin{pmatrix}
1\\ 1
\end{pmatrix}
$
and 
$
\begin{pmatrix}
1\\ -1
\end{pmatrix}
$
to the vectors
$
\begin{pmatrix}
p_1\\ q_1
\end{pmatrix}
$
and 
$
\begin{pmatrix}
p_2\\ q_2
\end{pmatrix}
$,
respectively.
The matrix $A$ is given by
\[
A=
\frac{1}{2}
\begin{pmatrix}
p_1+p_2 & p_1-p_2\\
q_1+q_2 & q_1-q_2
\end{pmatrix},
\]
and hence we have
\[
\pm R=AR_0A^{-1}=
\begin{pmatrix}
\frac{-1}{2}(p_1q_2+p_2q_1) & p_1p_2\\
-q_1q_2 & \frac{1}{2}(p_1q_2+p_2q_1)
\end{pmatrix}.
\]
Thus
\[
\begin{pmatrix}
p\\
q
\end{pmatrix}
=
\pm
R
\begin{pmatrix}
0\\
1
\end{pmatrix}
=
\pm
\begin{pmatrix}
p_1p_2\\
\frac{1}{2}(p_1q_2+p_2q_1)
\end{pmatrix}.
\]
Hence we have $p=p_1p_2$.

Observe that the involution on $\Conways$ induced by $R_0$
is the reflection in the circle that is the union 
of an arc of slope $1/1$ and an arc of slope $-1/1$.
Hence the restriction of the extra involution $g_R$ to $\Conways$
is the the reflection in the circle that is the union 
of an arc of slope $q_1/p_1$ and an arc of slope $q_2/p_2$.
Hence $\Fix(g_R)$ is a circle which intersects $K(r)$ in two points.
This confirms that $g_R$ is a strong inversion 
if and only if $K(r)$ is a knot 
(cf. Figures \ref{extra-symmetry}(1) and (3)).
So, we assume $K(r)$ is a knot.
Then the two extra meridian pairs that arise from the strong inversion $g_R$
are represented by 
an arc of slope $q_1/p_1$ and an arc of slope $q_2/p_2$, respectively.

If an extra meridian pair $\{m_1,m_2\}$ corresponds to an arc of slope $q_i/p_i$,
then the element $\omega(m_1,m_2)\in H_1(M(K(r)))<\Piorb(r)$ is represented 
the simple loop $\alpha_{q_i/p_i}$.
Hence we have
\[
\omega(m_1,m_2) = [\alpha_{q_i/p_i}] 
=p_i[\tilde\alpha_0]+q_i[\tilde\alpha_{\infty}]
=p_i[\tilde\alpha_0]
\in
H_1(M(K(r))
\cong \langle \tilde\alpha_{0} \ | \
p[\tilde\alpha_0]\rangle.
\]
We can observe that if $K(r)$ is hyperbolic then $1< p_i < p_1p_2$
(see Figure \ref{Farey-Symmetry}(2)).
So $\omega(m_1,m_2)$ is not a generator of $H_1(M(K(r))$.
This completes the proof of Proposition \ref{prop:extra-meridian-pair}
for the case where the condition (2) in Lemma \ref{lem:conti-fract-expansion2}
is satisfied.

\medskip
Thus we have proved that all possible parabolic generating pairs
listed in Propositions \ref{prop:strong-invertion_1} and \ref{prop:strong-invertion_2},
except for the upper/lower meridian pairs,
are not generating pairs.
This completes the proof of Theorem \ref{main-theorem2}(1).

\section{Parabolic generating pairs of Heckoid groups}
\label{sec:Heckoid-orb}

In this section, we prove Theorem \ref{main-theorem2}(2)
which asserts 
that every Heckoid group has a unique parabolic generating pair
up to equivalence.

We first give an explicit description of the weighted graph
in Figure \ref{generating-pairs} 
representing the Heckoid orbifolds
(cf. \cite[Definition 3.4]{AOPSY}, \cite[Section 5]{Lee-Sakuma_2013}).

\begin{definition}
\label{def:Heckoid-orbifold}
{\rm
(1)
For $r\in\QQ$ and for an integer $n\ge 2$,
$\orbm_{0}(r;n)$ denotes the orbifold pair determined by the weighted graph
$(S^3,K(r)\cup\tau_-,w_0)$, where $w_0$ is given 
\[
w_0(K(r))=\infty, \quad w_0(\tau_-)=n. 
\]

(2)
For $r=q/p\in\QQ$ with $p$ odd and an odd integer $m\ge 3$,
$\orbm_{1}(r;m)$ denotes the orbifold pair determined by the weighted graph
$(S^3,K(r)\cup\tau_-,w_1)$, where $w_1$ is given by the following rule.
Let $J_1$ and $J_2$ be the edges of the graph $K(r)\cup \tau_-$  distinct from $\tau_-$.
Then
\[
w_1(J_1)=\infty, \quad w_1(J_2)=2, \quad w_1(\tau_-)=m.
\]

(3)
For $r\in\QQ$ and an odd integer $m\ge 3$,
$\orbm_{2}(r;m)$ denotes the orbifold pair determined by the weighted graph
$(S^3,K(r)\cup\tau_+\cup\tau_-,w_2)$, where $w_2$ is given by the following rule.
Let $J_1$ and $J_2$ be unions of two mutually disjoint edges of the graph
$K(r)\cup\tau_+\cup\tau_-$ distinct from $\tau_{\pm}$,
such that
both $J_1$ and $J_2$ are preserved by the \lq vertical involution' $f$ of $K(r)$.
Then 
\[
w_2(J_1)=\infty, \quad w_2(J_2)=2, \quad w_2(\tau_+)=2, \quad w_2(\tau_-)=m.
\]
}
\end{definition}

In Definition \ref{def:Heckoid-orbifold}(3),
the {\it vertical involution} of $(S^3,K(r))$
is an involution of $(S^3,K(r))$ which preserves 
both $(\threeball,t(\infty))$ and $(\threeball,t(r))$
and whose restriction to the common boundary $\Conways$
is given by the $\pi$-rotation of $\RR^2$ about the point $(1/2,0)$
(see \cite[Figure 4]{AOPSY}).

We first calculate the orientation-preserving isometry group of 
the Heckoid orbifolds.
To this end, we recall the spherical dihedral orbifold $\OO(r;d_+,d_-)$
introduced in \cite[Theorem 4.1]{AOPSY},
that is represented by a weighted graph  
$(S^3,K(r)\cup\tau_+\cup\tau_-,w)$,
where $d_+$ and $d_-$ are mutually prime positive integers $d_+$ and $d_-$,
and $w$ is given by the rule.
\[
w(K(r))=2, \quad w(\tau_+)=d_+, \quad w(\tau_-)=d_-.
\]

\begin{proposition}
\label{prop:isometry-group-Heckoid}
The orientation-preserving isometry group of the
Heckoid orbifolds are given by the following fromula.
\[
\Isom^+(\orbm_{0}(r;n))  \cong  (\ZZ_2)^2,\quad
\Isom^+(\orbm_{1}(r;m)) \cong  \ZZ_2, \quad
\Isom^+(\orbm_{2}(r;m))  \cong  \ZZ_2.
\]
\end{proposition}
\begin{proof}
Let $\orbm$ be a Heckoid orbifold.
Then the isometry group $\Isom^+(\orbm)$ is finite, because 
the hyperbolic structure of $\orbm$
is geometrically finite.
Let $\orbo$ be the orbifold obtained from $\orbm$
by the order $2$ orbifold surgery,
namely the orbifold represented by the weighted graph
obtained from that representing the Heckoid orbifold $\orbm$,
by replacing the label $\infty$ with $2$
(cf. \cite[Definition 6.3]{AOPSY}).
Then $\OO$ is the spherical dihedral orbifold
$\OO(r;1,n)$, $\OO(r;1,m)$ or $\OO(r;2,m)$
according as
$\orbm=\orbm_0(r;n)$, $\orbm_1(r;m)$, or $\orbm_2(r;m)$.
The action of $\Isom^+(\orbm)$ descends
faithfully to an orientation-preserving finite group action
on the spherical orbifold $\OO$.
By the orbifold theorem (see \cite{BLP}), we may assume 
the action is an isometric action on the spherical orbifold.
Thus $\Isom^+(\orbm)$ is identified with the subgroup
of the spherical isometry group $\Isom^+(\OO)$ consisting of those isometries
which map the union of the edges of the singular set
that arise from the parabolic locus of $\orbm$.
The orientation-preserving isometry groups of the spherical dihedral orbifolds are given by \cite[Proposition 12.6]{AOPSY},
which together with the above fact enables us to obtain the desired isomorphisms.
\end{proof}

Let $\orbm=(M_0,P)$ be a Heckoid orbifold,
and identify $(M_0,P)$ with a subspace of the hyperbolic orbifold 
$M=\HH^3/\Gamma$ ($\Gamma=\pi_1(\orbm)$)
by the isomorphism (\ref{formula:convex-core}) in Section \ref{sec:intro}.
Let $\{\alpha, \beta\}$ be a parabolic generating pair of 
the Heckoid group $\pi_1(\orbm)$,
and let $h$ be the inverting elliptic element for $\{\alpha, \beta\}$,
i.e., the $\pi$-rotation about the geodesic $\eta$ joining the parabolic fixed points
of $\alpha$ and $\beta$.
Let $\bar\eta$ be the image of $\eta$ in $M$.
Then, by Proposition \ref{prop:strong-involution_2}, 
$\{\alpha,\beta\}$ is represented by $\bar\eta\cap M_0$.

\begin{figure}
	\centering
	\includegraphics[width=10cm]{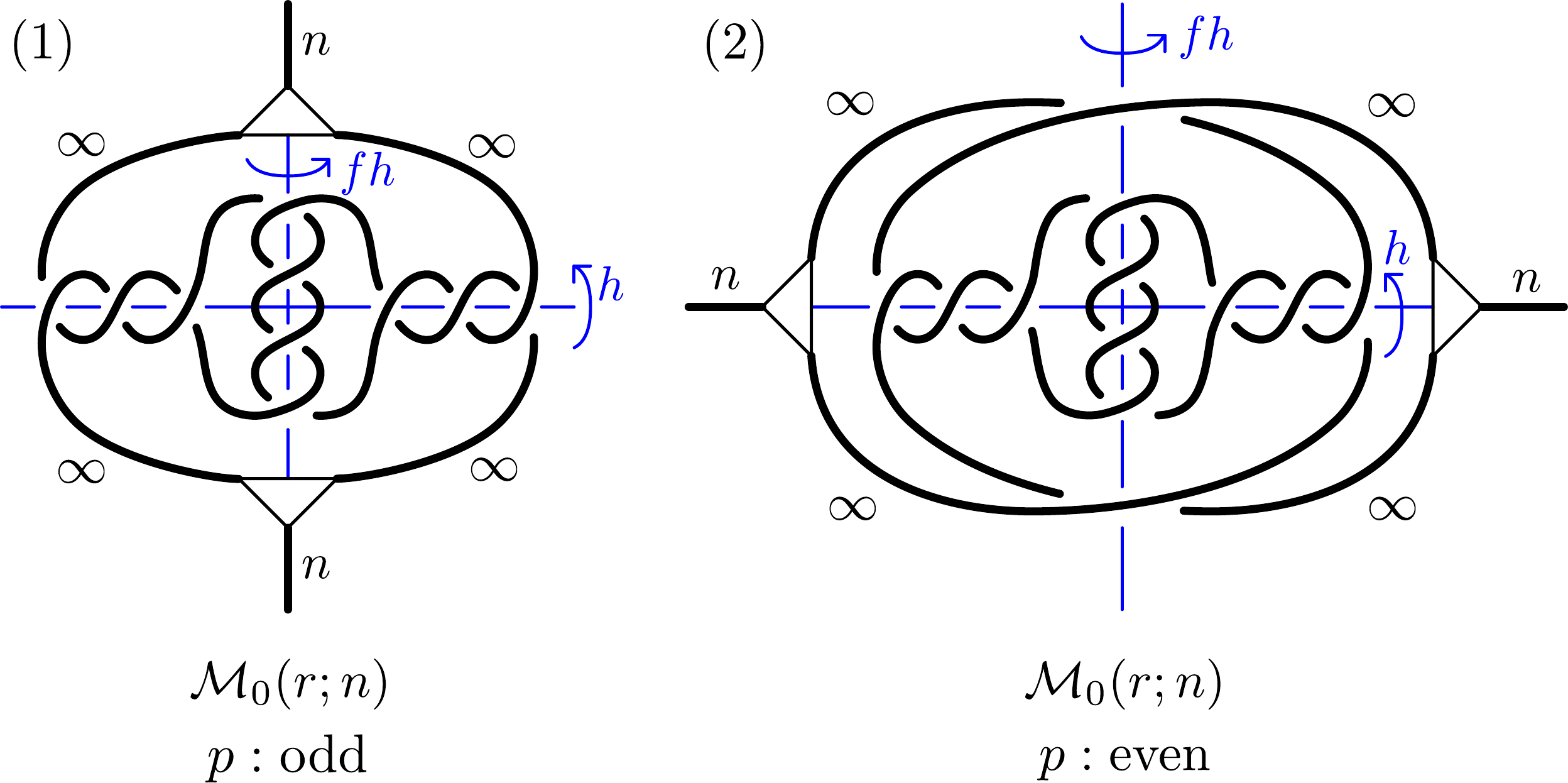}
	\caption{The inverting elliptic element for any generating pair
	induces the involution $h$ of $\orbm_0(r;n)$ in the figure. 
	According to whether $p$ is odd (1) or even (2),
	$\Fix(h)$ contains two or one geodesic paths joining 
	the parabolic locus $P$ to itself.}
	\label{Symmetry-Heckoid0}
\end{figure}

\smallskip

Case 1. $\orbm=\orbm_0(r;n)$ with $r=q/p$ and $n\ge 2$.
Then $\Isom^+(\orbm)\cong (\ZZ_2)^2$ 
is illustrated in
Figure \ref{Symmetry-Heckoid0}(1) or (2)
according to whether $p$ is odd or even.
Since the singular set of $\orbm$ consists of a single arc of index $n$
which joins the two components of $\closure(\partial M_0-P)$,
the inverting elliptic element $h$ does not belong to $\Gamma$,
and so it descends to an isometric involution of $\orbm$,
which we continue to denote by $h$.
Since its fixed point set
$\Fix(h)$ contains the geodesic path $\bar\eta$
joining the parabolic loci $P_{\alpha}$ and $P_{\beta}$,
the involution $h$ on $\orbm$ 
must be equivalent to the involution $h$ in Figure \ref{Symmetry-Heckoid0}(1) or (2)
according to whether $p$ is odd or even.

\smallskip

Subcase 1.1. $p$ is odd.
By Proposition \ref{prop:classification-fuchsian-case},
we have only to treat the case $p\ge 3$.
Note that $\Fix(h)$ of the involution $h$ in Figure \ref{Symmetry-Heckoid0}(1)
contains two geodesic paths which join 
the parabolic locus $P$ to itself,
namely $\tau_+$ and $\tau_+^c$,
the images of the tunnels $\tau_+$ and $\tau_+^c$ for $K(r)$
in $\orbm=\orbm_0(r;n)$
(see Figure \ref{Symmetry-Heckoid0}(1)).
If $\bar\eta=\tau_+$,
then $\{\alpha,\beta\}$  is equivalent to the standard 
parabolic generating pair in Figure \ref{generating-pairs}(2).
We show that $\bar\eta$ cannot be $\tau_+^c$.
Suppose on the contrary that $\bar\eta=\tau_+^c$.
Then the natural epimorphism
from $\pi_1(\orbm_0(r;n))$ onto the $2$-bridge knot group $G(r)$
maps the pair $\{\alpha,\beta\}$ to the long upper meridian pair of $G(r)$.
Since $p\ge 3$, the long upper meridian pair of $G(r)$ is not a generating pair
by Lemma \ref{lem:invariant-omega} and Proposition \ref{prop:long-meridian-pair}(1).
This contradicts the assumption that
$\{\alpha,\beta\}$ is a generating pair of $\pi_1(\orbm_0(r;n))$. 
Hence, $\pi_1(\orbm_0(q/p;n))$ with $p$ odd
has a unique parabolic generating pair.

\smallskip
Subcase 1.2. $p$ is even.
Then we can see from Figure \ref{Symmetry-Heckoid0}(2) that
$\tau_+$ is the unique geodesic path
contained in $\Fix(h)$ 
of the involution $h$ in Figure \ref{Symmetry-Heckoid0}(2)
which joins the parabolic locus $P$ to itself.
Hence the pair $\{\alpha,\beta\}$ is represented by $\tau_+$,
and so it is equivalent to
the standard parabolic generating pair in Figure \ref{generating-pairs}(2).
Hence, $\pi_1(\orbm_0(q/p;n))$ with $p$ even
has a unique parabolic generating pair.

\begin{figure}
	\centering
	\includegraphics[width=4cm]{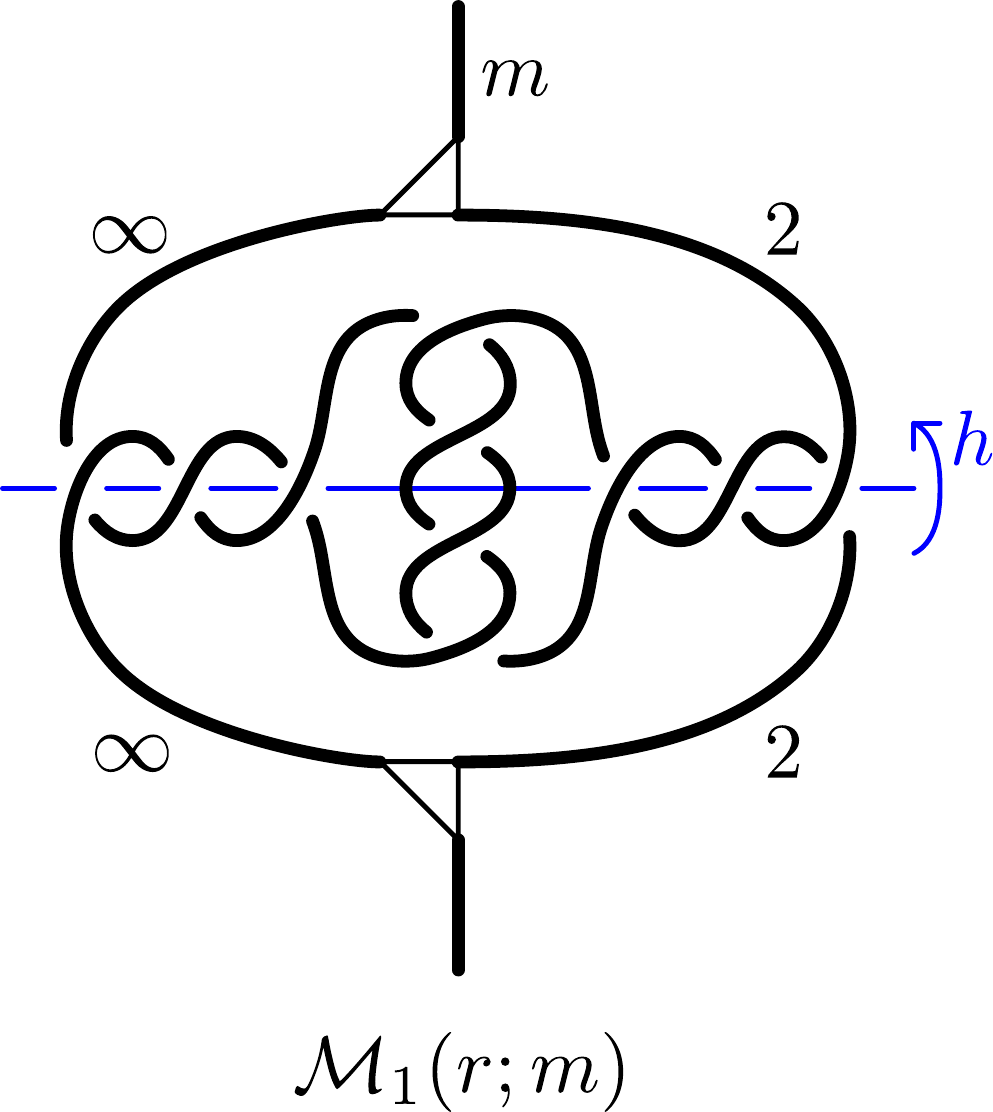}
	\caption{The inverting elliptic element for any generating pair
	induces the involution $h$ of the orbifold $\orbm_1(r;m)$ in the figure. 
	$\Fix(h)$ contains two geodesic paths joining 
	the parabolic locus $P$ to itself.}
	\label{Symmetry-Heckoid1}
\end{figure}

\smallskip

Case 2. $\orbm=\orbm_1(r;m)$ with $r=q/p$ ($p$: odd) and $m\ge 3$ odd.
Then the singular set of $\orbm$ contains a unique edge
of even index (actually $2$) and it joins the two components of $\closure(\partial M_0-P)$.
Hence the inverting elliptic element $h$ 
does not belong to $\Gamma$,
and so it descends to an isometric involution of $\orbm$,
which we continue to denote by $h$.
The involution $h$ generates $\Isom^+(\orbm)\cong \ZZ_2$ 
and it is as illustrated 
Figure \ref{Symmetry-Heckoid1}.
Observe that $\Fix(h)$ 
of the involution $h$ in Figure \ref{Symmetry-Heckoid1}
contains two geodesic paths which joins 
the parabolic locus $P$ to itself,
namely $\tau_+$ and $\tau_+^c$. 
If $\bar\eta=\tau_+$,
then $\{\alpha,\beta\}$  is equivalent to the standard 
parabolic generating pair in Figure \ref{generating-pairs}(3).
We show that $\bar\eta$ cannot be $\tau_+^c$.
Suppose on the contrary that $\bar\eta=\tau_+^c$.
Note that the natural epimorphism from $\pi_1(\orbm_1(r;m))$ 
onto the $\pi$-orbifold group $O(r)$ 
maps the pair $\{\alpha,\beta\}$
to $\{m_1, m_2m_1m_2^{-1}\}$,
where $\{m_1,m_2\}$ is the image in $O(r)$ of the long upper meridian pair of $G(r)$.
By Lemma \ref{lem:invariant-omega} and Proposition \ref{prop:long-meridian-pair}(1),
$\{m_1,m_2\}$ cannot generate $O(r)$.
(Here we use the fact that the conclusion of Lemma \ref{lem:invariant-omega}(2)
holds under the weaker condition that the image of $\{m_1,m_2\}$ in $O(K)$
generates $O(K)$.)
So $\{m_1, m_2m_1m_2^{-1}\}$ cannot generate $O(r)$, and 
hence $\{\alpha,\beta\}$ is not a generating pair, a contradiction.
Hence, $\pi_1(\orbm_1(r;m))$ has a unique parabolic generating pair.

\begin{figure}
\centering
\includegraphics[width=10cm]{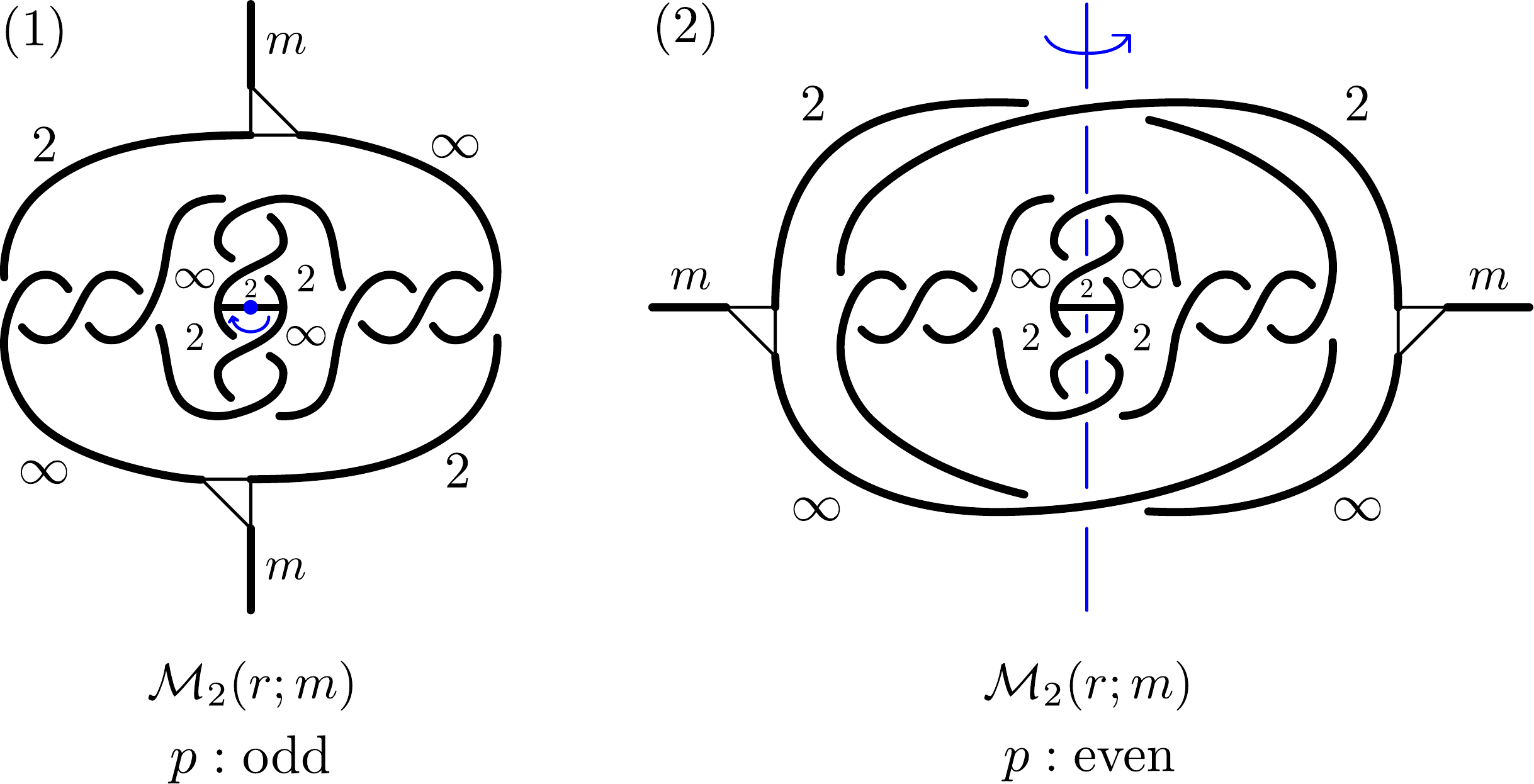}
\caption{$\orbm_2(r;m)$ admits a unique orientation-preserving isometric involution.
Its fixed point set is disjoint from the parabolic locus $P$. 
Hence the inverting elliptic element $h$
must be contained in $\pi_1(\orbm_2(r;m))$.}
\label{Symmetry-Heckoid2}
\end{figure}

\smallskip

Case 3. $\orbm=\orbm_2(r;m)$ with $m\ge 3$ odd.
Then $\Isom^+(\orbm)\cong \ZZ_2$ 
is as illustrated 
Figure \ref{Symmetry-Heckoid2}(1) or (2)
according to whether $p$ is odd or even.
Thus the fixed point set of the unique orientaion-preserving 
involution of $\orbm_2(r;m)$
does not contain a geodesic path connecting
the parabolic locus $P$ to itself, and so
the inverting elliptic element $h$ must belong to the Heckoid group
$\pi_1(\orbm_2(r;m))$.
Since $m$ is odd, $\bar \eta$ must be the upper tunnel $\tau_+$.
Hence the pair $\{\alpha,\beta\}$  is equivalent to
the standard parabolic generating pair in Figure \ref{generating-pairs}(4).
Thus $\pi_1(\orbm_2(r;m))$ has a unique parabolic generating pair.

\section{Applications to epimorphisms between $2$-bridge knot groups
and degree one maps between $2$-bridge link exteriors}
\label{sec:application}

In \cite{Ohtsuki-Riley-Sakuma},
Ohtsuki, Riley, and Sakuma gave
a systematic construction of epimorphisms between $2$-bridge link groups.
In this section, we show that 
all epimorporphisms between
$2$-bridge knot groups essentially arise from their construction.

We first recall the result of \cite{Ohtsuki-Riley-Sakuma}.
Let $\RGP{r}$ be the group of automorphisms of the Farey tessellation 
$\Farey$ generated by the reflections in the Farey edges 
with an endpoint $r$.
It should be noted that $\RGP{r}$
is isomorphic to the infinite dihedral group
and the region bounded by two adjacent Farey edges
with an endpoint $r$ is a fundamental domain
for the action of $\RGP{r}$ on $\HH^2$.
Let $\RGPP{r}$ be the group generated by $\RGP{\infty}$ and $\RGP{r}$.
When $r\in \QQ - \ZZ$,
$\RGPP{r}$ is equal to the free product $\RGP{\infty}*\RGP{r}$,
having a fundamental domain shown in Figure \ref{fig.fd}.
(Otherwise, $\RGPP{r}$ is equal to $\RGP{\infty}$ or the group
generated by the reflections in all Farey edges
according to whether $r=\infty$ or $r\in\ZZ$.)
It should be noted that Schubert's classification theorem of $2$-bridge links
says that two 2-bridge links $K(r)$ and $K(r')$ are equivalent
if and only if there is an automorphism of $\Farey$
which sends $\{\infty,r\}$ to $\{\infty,r'\}$.
Thus the conjugacy class of the group $\RGPP{r}$
in the automorphism group of $\Farey$ is uniquely determined by
the link $K(r)$.

\begin{figure}
\begin{center}
\includegraphics[width=7cm]{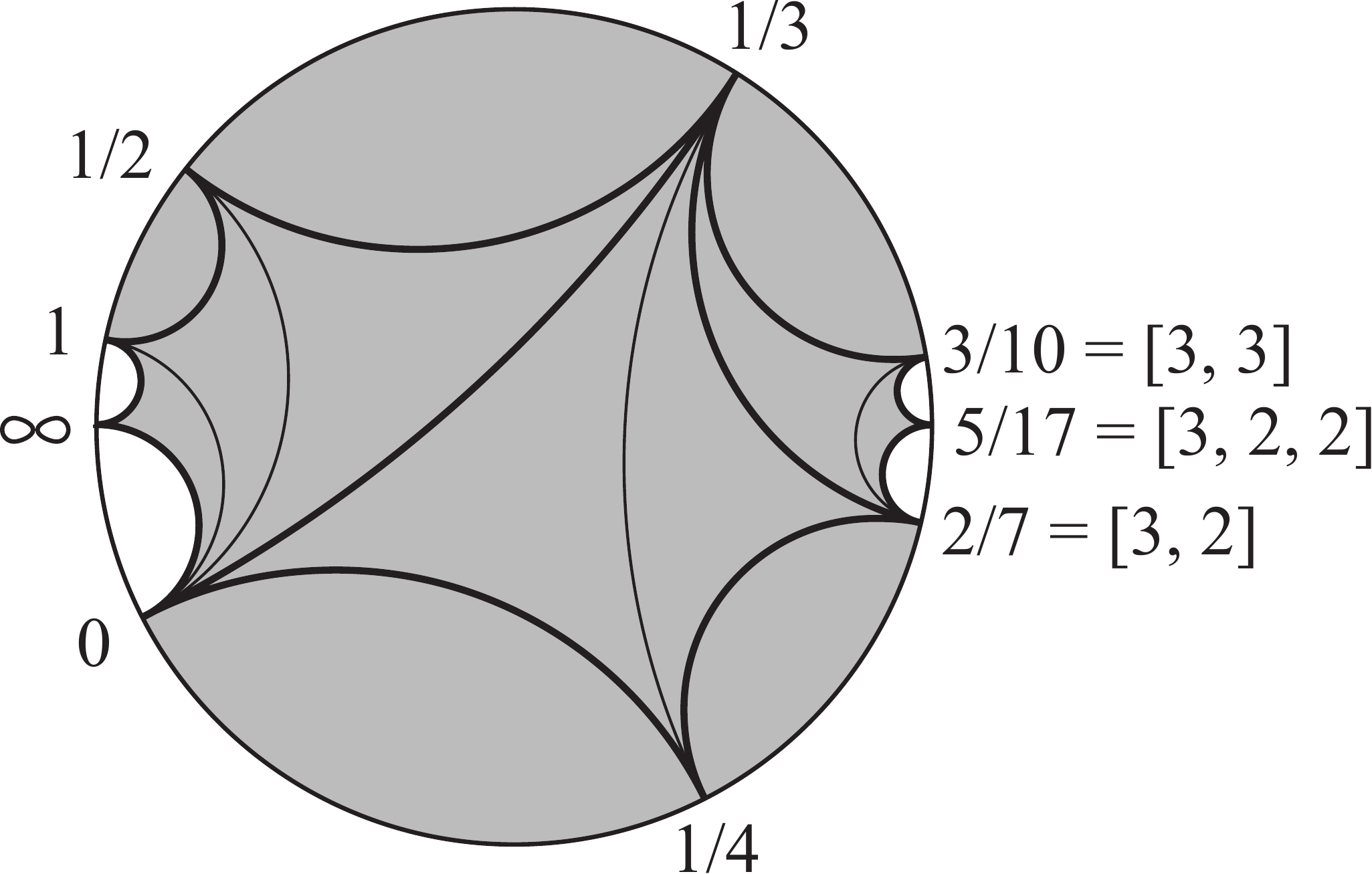}
\end{center}
\caption{
A fundamental domain of $\hat\Gamma_r$ in the
Farey tessellation (the shaded domain) for $r=5/17=[3,2,2]$.}
\label{fig.fd}
\end{figure}

The following result is a consequence of
\cite[Theorem 1.1]{Ohtsuki-Riley-Sakuma}
and the well-known fact that there is a upper-meridian-preserving
isomorphism $G(r)\cong G(r+1)$ for any $r\in\QQQ$
(cf. \cite[Proposition 2.1(1.a)]{AOPSY}).

\begin{enumerate}
\item[(A)]
{\it 
There is an epimorphism from a $2$-bridge link group $G(\tilde r)$ to 
a $2$-bridge link group $G(r)$, if $\tilde r$ or $\tilde r+1$ belongs to
the $\RGPP{r}$-orbit of $r$ or $\infty$.
Moreover, we can choose the epimorphism so that it is upper-meridian-pair-preserving,
namely, it sends the upper meridian pair of $G(\tilde r)$ to that of $G(r)$.
}
\end{enumerate}

Lee and Sakuma proved the following converse to the above result
(\cite[Theorem 2.4]{Lee-Sakuma_2012}),
by solving a certain word problem for the $2$-bridge link groups,
using the small cancellation theory. 

\begin{enumerate}
\item[(B)]
{\it 
There is an upper-meridian-pair-preserving epimorphism
from a $2$-bridge link group $G(\tilde r)$ to 
a $2$-bridge link group $G(r)$
if and only if
$\tilde r$ or $\tilde r+1$
belongs to the $\RGPP{r}$-orbit of $r$ or $\infty$.
}
\end{enumerate}

On the other hand,
it was proved by
Boileau, Boyer, Reid, and Wang that 
any epimorphism from a hyperbolic $2$-bridge knot group $G(\tilde r)$
onto a non-trivial knot group $G(K)$ is induced by
a non-zero degree map $S^3-K(r)\to S^3-K$
and that $K$ is necessarily a two-bridge knot
(\cite[Corollary~1.3]{BBRW}).
This in particular implies the following.

\begin{enumerate}
\item[(C)]
{\it 
Any epimorphism from a hyperbolic $2$-bridge knot group $G(\tilde r)$ to 
a nontrivial $2$-bridge knot group $G(r)$
maps the upper meridian pair of $G(\tilde r)$ to
a pair consisting of peripheral elements.
In particular, if $G(r)$ is also a hyperbolic $2$-bridge knot group,
then the image in $G(r)$ of the upper meridian pair of $G(\tilde r)$
is a parabolic generating pair of $G(r)$.
}
\end{enumerate}

Furthermore, 
Gonzal\'ez-Ac\~una and Ram\'inez~ \cite[Theorem 1.2]{Gonzalez-Raminez1} 
completely determined the 2-bridge knot groups that have epimorphisms
onto a $(2,p)$ torus knot group $G(1/p)$.
Their result can be reformulated as follows.

\begin{enumerate}
\item[(D)]
{\it 
There is 
an epimorphism from a $2$-bridge knot group $G(\tilde r)$ to 
a $(2,p)$ torus knot group $G(r)$ with $r=1/p$ ($p>1$ odd)
if and only if 
$\tilde r$ or $\tilde r+1$
belongs to the $\hat\Gamma_r$-orbit of $r=1/p$ or $\infty$.
}
\end{enumerate}

In fact, if there is an epimorphism $\varphi:G(\tilde r)\to G(1/p)$, then
the composition of $\varphi$ with the natural epimorphism 
$G(1/p)\to \ZZ_2*\ZZ_p$ determines an epimorphism
$G(\tilde r)\to\ZZ_2*\ZZ_p$.
Thus $\tilde r$ satisfies the condition (iii) of
\cite[Theorem 1.2]{Gonzalez-Raminez1},
which is equivalent to the condition (v) of the theorem.
The latter conidition in turn is equivalent to the condition 
that $\tilde r$ or $\tilde r+1$
belongs to the $\hat\Gamma_r$-orbit of $r=1/p$ or $\infty$
by \cite[Proposition 5.1]{Ohtsuki-Riley-Sakuma}.

By using the above results (A)$\sim$(D)
and Theorem \ref{main-theorem2}(1),
we obtain the following complete characterisation 
of epimorphisms between $2$-bridge knot groups. 

\begin{theorem}
\label{thm:cor-epimorohism}
There is an epimorphism from a $2$-bridge knot group $G(\tilde r)$ to 
a $2$-bridge knot group $G(r)$ with $r=q/p$,
if and only if one of the following conditions holds.
\begin{enumerate}
\item
$\tilde r$ or $\tilde r+1$
belongs to the $\hat\Gamma_r$-orbit of $r$ or $\infty$.
\item
$\tilde r$ or $\tilde r+1$
belongs to the $\hat\Gamma_{r'}$-orbit of $r'$ or $\infty$,
where $r'= q'/ p$
with $q q'\equiv 1 \pmod p$.
\end{enumerate}
\end{theorem}

\begin{proof}
Recall the well-known fact that 
if $r$ and $r'$ are as in (2) in the theorem then
there is an isomorphism $G(r)\to G(r')$
which brings the upper meridian pair of $G(r)$
to the lower meridian pair of $G(r')$
(cf. \cite[Proposition 2.1(1.b)]{AOPSY}).
The if part of the theorem is a consequence of the result (A) and this fact.

So we prove the only if part.
If $K(r)$ is a trivial knot, then every $\tilde r$ satisfies the condition
(1) or (2), and so the result holds trivially.
If $K(r)$ is a nontrivial torus knot, then it is nothing other than the result (D).
Thus we may assume $K(r)$ is a hyperbolic knot.
Let $\varphi:G(\tilde r)\to G(r)$ be an epimorphism and
$\{\tilde\alpha,\tilde\beta\}$ the upper meridian pair of $G(\tilde r)$.
Set $\{\alpha,\beta\}:=\{\varphi(\tilde\alpha), \varphi(\tilde\beta)\}$.
Then, by the result (C),
$\{\alpha,\beta\}$ is a parabolic generating pair of $G(r)$.
Thus, by Theorem~\ref{main-theorem2}(1),
it is equivalent to the upper or lower meridian pair of $G(r)$.
If $\{\alpha,\beta\}$ is equivalent to the upper-meridian pair,
then the result (B)
implies that the condition (1) holds.
If $\{\alpha,\beta\}$ is equivalent to the lower meridian pair of $G(r)$,
its image in $G(r')$ by the isomorphism $G(r)\to G(r')$, 
described at the beginning of the proof,
is the upper meridian pair of $G(r')$.
Thus the composition $G(\tilde r)\to G(r)\cong G(r')$
is an upper-meridian-pair-preserving epimorphism.
Hence the result (B) implies that 
$\tilde r$ satisfies the condition (2).
This completes the proof of Theorem~\ref{thm:cor-epimorohism}.
\end{proof}

\begin{remark}
\label{rem:determination-epi}
{\rm
(1) We can describe all epimorphisms 
between given two hyperbolic $2$-bridge knot groups as follows.
For each $r\in\QQQ$, fix a one-relator presentation 
$G(r)=\langle a,b \ | \ u_r\rangle$ as in \cite[Section 3]{Lee-Sakuma_2012},
where the ordered pair $(a,b)$ represents the upper-meridian pair.
If $\tilde r$ belongs to the $\hat\Gamma_r$-orbit of $r$ or $\infty$,
then the identity map on the free group $F(a,b)$ with free basis $\{a,b\}$
descends to an (upper-meridian-pair-preserving) epimorphsim $G(\tilde r)\to G(r)$:
we call it
the {\it ORS epimorphism} from $G(\tilde r)$ to $G(r)$
(cf. \cite[Proof of Theorem 1.1 in p.428]{Ohtsuki-Riley-Sakuma}).
By the proof of Theorem \ref{thm:cor-epimorohism}
and \cite[Proof of Main Theorem 2.4 in pp.364--365]{Lee-Sakuma_2012},
we can see that 
any epimorphism $G(\tilde r) \to G(r)$ 
between given two hyperbolic $2$-bridge knot groups
is equal to a composition 
\[
G(\tilde r)\cong G(\tilde r^*)\to G(r^*)\cong G(r)
\]
of two isomorphisms between $2$-bridge knot groups
and an ORS epimorphism $G(\tilde r^*)\to G(r^*)$.

(2) The above conclusion also holds 
for any epimorphism between hyperbolic $2$-bridge link groups
which are induced by a non-zero degree map between the exteriors of the links,
because such an epimorphism maps the upper-meridian pair to a parabolic generating pair and so we can apply Theorem~\ref{main-theorem2}(1) as in the proof of
Theorem \ref{thm:cor-epimorohism}.

(3) The results of 
Gonzal{\'e}z-Ac{\~u}na and Ram{\'i}nez \cite{Gonzalez-Raminez1, 
Gonzalez-Raminez2b, Gonzalez-Raminez2} imply that
the conclusion in (1) also holds when the target is a 
non-hyperbolic $2$-bridge knot $K(r)$ with $r=1/p$ ($p>1$ odd).
To see this, note that
$G(1/p)$ is the pullback (or the fiber product)
of the diagram $\ZZ\to \ZZ_{2p}\leftarrow \ZZ_2*\ZZ_p$
(see \cite[Example 2]{Hartley-Murasugi}).
In fact, this diagram together with
the abelianization $\xi:G(1/p)\to \ZZ$
and the natural epimorphism $\rho:G(1/p)\to \ZZ_2*\ZZ_p$
form a pullback diagram.
Now let $\varphi:G(\tilde r)\to G(1/p)$ be an epimorphism.
For simplicity, we assume that $\tilde r$ belongs to the $\hat\Gamma_r$-orbit 
of $r$ or $\infty$, and let $\varphi_0:G(\tilde r)\to G(1/p)$
be the ORS epimorphism.
Set $\tilde\xi=\xi\circ\varphi$, $\tilde\xi_0=\xi\circ\varphi_0$, 
$\tilde\rho=\rho\circ\varphi$, $\tilde\rho_0=\rho\circ\varphi_0$.
Then $\varphi$ and $\varphi_0$ are uniquely determined by the pairs 
$(\tilde\xi, \tilde\rho)$ and $(\tilde\xi_0, \tilde\rho_0)$, respectively. 
By the argument in \cite[the last part of the proof of Corollary 1.3]{BBRW}
based on \cite[Theorem 18]{Gonzalez-Raminez2b},
we see either (i) $\tilde\xi=\tilde\xi_0$ and $\tilde\rho=\tilde\rho_0$
modulo post composition of an inner-automorphism of $\ZZ_2*\ZZ_p$ 
or (ii) $\tilde\xi=h_1\circ\tilde\xi_0$ and $\tilde\rho=h_2\circ\tilde\rho_0$
modulo post composition of an inner-automorphism of $\ZZ_2*\ZZ_p$,
where $h_1=-1_{\ZZ}$ is the automorphism of $\ZZ$ defined by $h_1(z)=-z$
and $h_2$ is the automorphism of $\ZZ_2*\ZZ_p$ 
defined by $h_2=1_{\ZZ_2}*(-1_{\ZZ_{p}})=(-1_{\ZZ_2})*(-1_{\ZZ_{p}})$.
In the first case, $\varphi=\varphi_0$ modulo post composition of an 
inner-automorphism of $G(1/p)$.
In the second case, $\varphi=h\circ \varphi_0$
modulo post composition of an 
inner-automorphism of $G(1/p)$,
where $h$ is the automorphism of $G(1/p)$
induced by the strong inversion of $K(1/p)$.
}
\end{remark}

\medskip
We do not know if Theorem \ref{thm:cor-epimorohism} holds 
for $2$-bridge links,
because we do not know if the result (C) holds for 
$2$-component $2$-bridge links.
In fact, there are epimorphisms between $2$-bridge link groups
which are induced by degree $0$ maps between the link complements
(see \cite[Remark 6.3 and Proof of Corollary 8.1]{Ohtsuki-Riley-Sakuma}),
in contrast to \cite[Corollary~1.3]{BBRW}.
However, 
we can obtain a condition for the existence of
a degree one map between hyperbolic $2$-bridge link complements.
(We thank the referee for suggesting this to us.)
From now on, we work with link exteriors instead of link complements,
and by a {\it degree one map} between the exteriors of two links $L_1$ and $L_2$
in $S^3$,
we mean a map $f:(E(L_1),\partial E(L_1))\to (E(L_2),\partial E(L_2))$
such that $f_*:H_3(E(L_1),\partial E(L_1)) \to H_3(E(L_2),\partial E(L_2))$
carries the fundamental class  $[E(L_1),\partial E(L_1)]$ to
$\pm[E(L_2),\partial E(L_2)]$.
To describe the result, recall that any ORS epimorphism
$G(\tilde r)\to G(r)$ is induced by
(the restriction to the link exteriors of)
a \lq\lq branched fold map'' between 
$2$-bridge links (see \cite[Theorem 6.1]{Ohtsuki-Riley-Sakuma}),
whose degree, $\deg(\tilde r,r)$,
is calculated from continued fraction expansions of $\tilde r$ and $r$
(see \cite[Proposition 6.2(2) and Remark 6.3(1)]{Ohtsuki-Riley-Sakuma}).

\begin{theorem}
\label{thm:degree-one}
There is a degree one map 
$f:E(K(\tilde r))\to E(K(r))$ with $r=q/p$
between hyperbolic $2$-bridge link exteriors,
if and only if one of the following conditions holds.
\begin{enumerate}
\item
$\tilde r$ or $\tilde r+1$
belongs to the $\hat\Gamma_r$-orbit of $r$ or $\infty$,
and $\deg(\tilde r,r)$ or $\deg(\tilde r+1,r)$, accordingly, is equal to $\pm 1$.
\item
$\tilde r$ or $\tilde r+1$
belongs to the $\hat\Gamma_{r'}$-orbit of $r'$ or $\infty$,
where $r'= q'/ p$
with $q q'\equiv 1 \pmod p$,
and $\deg(\tilde r,r')$ or $\deg(\tilde r+1,r')$, accordingly, is equal to $\pm 1$.
\end{enumerate}
\end{theorem} 

\begin{proof}
The if part follows from the description above
and \cite[Theorem 6.1 and Proposition 6.2]{Ohtsuki-Riley-Sakuma}.
To prove the only if part,
let $f:E(K(\tilde r))\to E(K(r))$ be a degree one map
between hyperbolic $2$-bridge link exteriors.
Then $f$ induces an epimorphism
$f_*:G(\tilde r)\to G(r)$,
and it maps the upper meridian pair of $G(\tilde r)$
to a parabolic generating pair of $G(r)$.
Thus, as in the proof of Theorem \ref{thm:cor-epimorohism},
we see by using Theorem~\ref{main-theorem2} that
one of the conditions (1) and (2) in Theorem~\ref{main-theorem2} holds.
Moreover, as noted in Remark \ref{rem:determination-epi}(2),
the epimorphism $f_*$ is equal to an ORS epimorphism
modulo pre/post-compositions of isomorphisms
of $2$-bridge link groups.
Since $E(K(r))$ is aspherical
and since the peripheral subgroups are malnormal, 
the map $f$ is properly homotopic to the 
branch-fold map 
given by \cite[Theorem 6.1]{Ohtsuki-Riley-Sakuma},
modulo pre/post-compositions of homeomorphisms between $2$-bridge link exteriors.
Hence the degree of $f$ is equal to
$\pm\deg(\tilde r,r)$, $\pm\deg(\tilde r+1,r)$, $\pm\deg(\tilde r,r')$ 
or $\pm\deg(\tilde r+1,r')$
accordingly.
Since $\deg(f)=\pm1$, 
one of the conditions of Theorem \ref{thm:degree-one} holds. 
\end{proof}

\nocite{*}
\bibliographystyle{rendiconti}
\bibliography{ParabolicGP_rendiconti}

\providecommand{\bysame}{\leavevmode\hbox to3em{\hrulefill}\thinspace}
\begin{thebibliography}{10}

\bibitem{Adams1}
{\sc C.~Adams}, \emph{Hyperbolic $3$-manifolds with two generators},
  Comm.~Anal.~Geom. \textbf{4} (1996), 181--206.

\bibitem{Adams-Reid}
{\sc C.~Adams and A.~Reid}, \emph{Unknotting tunnels in two-bridge knot and
  link complements}, Comment.~Math.~Helv. \textbf{71} (1996), 617--627.

\bibitem{Agol}
{\sc I.~Agol}, \emph{The classification of non-free $2$-parabolic generator
  {K}leinian groups}, Slides of talks given at Austin AMS Meeting and Budapest
  Bolyai conference, July 2002, Budapest, Hungary.

\bibitem{Agol-Liu}
{\sc I.~Agol and Y.~Liu}, \emph{Presentation length and simon's conjecture},
  J.~Amer.~Math.~Soc. \textbf{25} (2012), 151--187.

\bibitem{Aimi}
{\sc S.~Aimi}, \emph{Parabolic generating pairs of hyperbolic $2$-bridge knot
  groups}, Master's thesis, Hiroshima University, 2016, in Japanese.

\bibitem{ASWY}
{\sc H.~Akiyohi, M.~Sakuma, M.~Wada, and Y.~Yamashita}, \emph{Punctured torus
  groups and $2$-bridge knot groups {{\rm(I)}}}, Lecture Notes in Mathematics,
  vol. 1909, Springer, Berlin, 2007.

\bibitem{AOPSY}
{\sc H.~Akiyoshi, K.~Ohshika, J.~Parker, M.~Sakuma, and H.~Yoshida},
  \emph{Classification of non-free {K}leinain groups generated by two parabolic
  transformations}, to appear in Trans.~Amer.~Math.~Soc., arXiv:2001.09564
  [math.GT].

\bibitem{Bleiler-Moriah}
{\sc S.~Bleiler and Y.~Moriah}, \emph{Heegaard splittings and branched
  coverings of $b^3$}, Math.~Ann \textbf{281} (1988), 531--543.

\bibitem{Boileau}
{\sc M.~Boileau}, \emph{Private communication}, May 2011.

\bibitem{BBRW}
{\sc M.~Boileau, S.~Boyer, A.~Reid, and S.~Wang}, \emph{Simon's conjecture for
  two-bridge knots}, Comm.~Anal.~Geom. \textbf{18} (2010), 121--143.

\bibitem{BLP}
{\sc M.~Boileau, B.~Leeb, and J.~Porti}, \emph{Geometrization of
  $3$-dimensional orbifolds}, Ann.~of Math. \textbf{162} (2005), 195--290.

\bibitem{BMP}
{\sc M.~Boileau, S.~Maillot, and J.~Porti}, \emph{Three-dimensional orbifolds
  and their geometric structures}, Panoramas et Synth{\`e}ses [Panoramas and
  Syntheses], vol.~15, Soci{\'e}t{\'e} Math{\'e}matique de France, Paris, 2003,
  viii+167pp.

\bibitem{Boileau-Porti}
{\sc M.~Boileau and J.~Porti}, \emph{Geometrization of $3$-orbifolds of cyclic
  type}, Ast{\'e}risque (2001), no.~272, Appendix A by Michael Heusener and
  Porti, 208pp.

\bibitem{Boileau-Zimmermann}
{\sc M.~Boileau and B.~Zimmermann}, \emph{The $\pi$-orbifold group of a link},
  Math.~Z. \textbf{200} (1989), 187--208.

\bibitem{Bonahon-Siebenmann}
{\sc F.~Bonahon and L.~Siebenmann}, \emph{New geometric splittings of classical
  knots, and the classification and symmetries of arborescent knots}, preprint,
  available from
  \url{http://www-bcf.usc.edu/~fbonahon/Research/Preprints/Preprints.html}.

\bibitem{BZH}
{\sc G.~Burde, H.~Zieschang, and M.~Heusener}, \emph{Knots}, third, fully
  revised and extended edition ed., De Gruyter Studies in Mathematics, vol.~5,
  De Gruyter, Berlin, 2014, xiv+417.

\bibitem{Conner-Raymond}
{\sc P.~E. Conner and F.~Raymond}, \emph{Manifolds with few periodic
  homeomorphisms}, Proceedings of the Second Conference on Compact
  Transformation Groups $(1971)$ (Amherst, Mass.), Univ.~Massachusetts,
  Springer, Berlin, 1972, Part II, Lecture Notes in Math., Vol.~{\bf 299},
  pp.~1--75.

\bibitem{CHK}
{\sc D.~Cooper, C.~Hodgson, and S.~Kerckhoff}, \emph{Three-dimensional
  orbifolds and cone-manifolds}, MSJ Memoirs, vol.~5, Mathematical Society of
  Japan, Tokyo, 2000, x+170 pp.

\bibitem{Epstein-Penner}
{\sc D.~Epstein and R.~Penner}, \emph{Euclidean decompositions of noncompact
  hyperbolic manifolds}, J. Differential Geom. \textbf{27} (1988), 67--80.

\bibitem{Futer}
{\sc D.~Futer}, \emph{Involutions of knots that fix unknotting tunnels}, J.
  Knot Theory Ramifications \textbf{16} (2007), 741--748.

\bibitem{Gonzalez-Raminez1}
{\sc F.~Gonzal{\'e}z-Ac{\~u}na and A.~Ram{\'i}nez}, \emph{Two-bridge knots with
  property {Q}}, Quart.~J. Math. \textbf{52} (2001), 447--454.

\bibitem{Gonzalez-Raminez2b}
{\sc F.~Gonzal{\'e}z-Ac{\~u}na and A.~Ram{\'i}nez}, \emph{Normal forms of
  generating pairs for fuchsian and related groups}, Math. Ann. \textbf{322}
  (2002), 207--227.

\bibitem{Gonzalez-Raminez2}
{\sc F.~Gonzal{\'e}z-Ac{\~u}na and A.~Ram{\'i}nez}, \emph{Epimorphisms of knot
  groups onto free products}, Topology \textbf{42} (2003), 1205--1227.

\bibitem{GLM}
{\sc C.~McA. Gordon, R.~A. Litherland, and K.~Murasugi}, \emph{Signatures of
  covering links}, Canadian J. Math. \textbf{33} (1981), 381--394.

\bibitem{Gueritaud}
{\sc F.~Gu{\'e}ritaud}, \emph{Geometrie hyperbolique effective et
  triangulations ideals canoniques en dimension $3$}, Ph.D. thesis,
  Universit{\'e} Paris-Sud, 2006, 11.

\bibitem{Gueritaud-Futer}
{\sc F.~Gu{\'e}ritaud}, \emph{On canonical triangulations of once-punctured
  torus bundles and two-bridge link complements}, Geom.~Topol. \textbf{10}
  (2006), 1239--1284, With an appendix by David Futer.

\bibitem{Hartley-Kawauchi}
{\sc R.~Hartley and A.~Kawauchi}, \emph{Polynomials of amphicheiral knots},
  Math.~Ann. \textbf{243} (1979), 63--70.

\bibitem{Hartley-Murasugi}
{\sc R.~Hartley and K.~Murasugi}, \emph{Homology invariants}, Canadian~J.~Math.
  \textbf{30} (1978), 655--670.

\bibitem{Hecke}
{\sc E.~Hecke}, \emph{{\"U}ber die {B}estimmung {D}irichletscher {R}eihen durch
  ihre {F}unktionalgleichung}, Math.~Ann. \textbf{112} (1936), 664--699.

\bibitem{Heusener-Porti}
{\sc M.~Heusener and J.~Porti}, \emph{Generating pair of $2$-bridge knot
  groups}, Geom.~Dedicata \textbf{151} (2011), 279--295.

\bibitem{HMJS}
{\sc J.~Hoste, J.~O. Mercado, and P.~Shanahan}, \emph{Remarks on {S}uzuki's
  knot epimorphism number}, J. Knot Theory Ramifications \textbf{28} (2019),
  13pp, 1950060.

\bibitem{Knapp}
{\sc A.~W. Knapp}, \emph{Doubly generated {F}uchsian groups}, Michigan Math.~J.
  \textbf{15} (1969), 289--304.

\bibitem{Lee-Sakuma_2012}
{\sc D.~Lee and M.~Sakuma}, \emph{Epimorphisms between $2$-bridge link groups:
  homotopically trivial simple loops on $2$-bridge spheres}, Proc.~London
  Math.~Soc. \textbf{104} (2012), 359--386.

\bibitem{Lee-Sakuma_2013}
{\sc D.~Lee and M.~Sakuma}, \emph{Epimorphisms from $2$-bridge link groups onto
  {H}eckoid groups {{\rm(I)}}}, Hiroshima Math.~J. \textbf{43} (2013),
  239--264.

\bibitem{Lee-Sakuma_2016}
{\sc D.~Lee and M.~Sakuma}, \emph{Parabolic generating pairs of genus-one
  $2$-bridge knot groups}, J. Knot Theory Ramifications \textbf{25} (2016), 21
  pp, 1650023.

\bibitem{Millichap-Worden}
{\sc C.~Millichap and W.~Worden}, \emph{Hidden symmetries and commensurability
  of $2$-bridge link complements}, Pacific J. Math. \textbf{285} (2016),
  453--484.

\bibitem{Morimoto-Sakuma}
{\sc K.~Morimoto and M.~Sakuma}, \emph{On unknotting tunnels for knots},
  Math.~Ann. \textbf{289} (1991), 143--167.

\bibitem{Ohtsuki-Riley-Sakuma}
{\sc T.~Ohtsuki, R.~Riley, and M.~Sakuma}, \emph{Epimorphisms between
  $2$-bridge link groups}, Geom.~Topol.~Monogr. \textbf{14} (2008), 417--450.

\bibitem{Riley1979}
{\sc R.~Riley}, \emph{An elliptical path from parabolic representations to
  hyperbolic structures}, Topology of low-dimensional manifolds (Proc.~Second
  Sussex Conf., 1977) (Chelwood Gate), Springer, Berlin, 1979, Lecture Notes in
  Math., Vol.~{\bf 722}, pp.~99--133.

\bibitem{Riley1992}
{\sc R.~Riley}, \emph{Algebra for {H}eckoid groups}, Trans.~Amer.~Math.~Soc.
  \textbf{334} (1992), 389--409.

\bibitem{Sakuma1}
{\sc M.~Sakuma}, \emph{On strongly invertible knots}, Algebraic and topological
  theories (Kinosaki, 1984), Kinokuniya, Tokyo, 1986, pp.~176--196.

\bibitem{Sakuma2}
{\sc M.~Sakuma}, \emph{The geometries of spherical {M}ontesinos links}, Kobe J.
  Math. \textbf{7} (1990), 167--190.

\bibitem{Sakuma-Weeks}
{\sc M.~Sakuma and J.~Weeks}, \emph{Examples of canonical decompositions of
  hyperbolic link complements}, Japan.~J. Math.\ (N.S.) \textbf{21} (1995),
  393--439.

\bibitem{Schubert}
{\sc H.~Schubert}, \emph{Knoten mit zwei {B}r\"cken}, Math.~Z. \textbf{65}
  (1956), 133--170.

\bibitem{Suzuki}
{\sc M.~Suzuki}, \emph{Epimorphisms between $2$-bridge knot groups and their
  crossing numbers}, Algebr.~Geom.~Topol. \textbf{17} (2017), 2413--2428.

\bibitem{Suzuki-Tran}
{\sc M.~Suzuki and A.~T. Tran}, \emph{Genera of two-bridge knots and
  epimorphisms of their knot groups}, Topology Appl. \textbf{242} (2018),
  66--72.

\bibitem{Zieschang}
{\sc H.~Zieschang}, \emph{Generators of free product with amalgamation of two
  infinite cyclic groups}, Math.~Ann. \textbf{227} (1977), 195--221.

\end{thebibliography}

\end{document}